\documentclass{article}
\usepackage[utf8]{inputenc}
\usepackage{amsthm}
\usepackage{amsmath}
\usepackage{amssymb}
\usepackage{enumitem}
\usepackage{bookmark}
\RequirePackage{geometry}
\usepackage{caption}
\usepackage{graphicx}
\usepackage{hyperref}
\usepackage[T1]{fontenc}
\usepackage{amsfonts}
\usepackage{braket}
\usepackage[l2tabu, orthodox]{nag}
\usepackage{pdf14}
\usepackage{xpatch}
\RequirePackage{tabto}
\usepackage{indentfirst}
\usepackage{amscd,stmaryrd,latexsym}

\newtheorem{thm}{Theorem}
\newtheorem{lem}{Lemma}[section]

\newtheorem{cor}{Corollary}
\newtheorem{Prop}{Proposition}[section]

\newtheorem*{theorem*}{Theorem}
\newtheorem*{corollary*}{Corollary}

\theoremstyle{definition}

\newtheorem{oss}{Remark}[section]

\newcommand{\R}{\mathbb{R}}
\newcommand{\N}{\mathbb{N}}

\def\eps{\mathop{\varepsilon}}

\DeclareMathAlphabet{\mathscr}{OT1}{pzc}{m}{it}

\begin{document} 
 
\title{\textbf{Extensions of Schoen--Simon--Yau and Schoen--Simon theorems via iteration \`{a} la De Giorgi}} 
\author{Costante Bellettini}
\date{}

\maketitle

\begin{abstract}
We give an alternative proof of the Schoen--Simon--Yau curvature estimates and associated Bernstein-type theorems (1975), and extend the original result by including the case of $6$-dimensional (stable minimal) immersions.
The key step is an $\eps$-regularity theorem, that assumes smallness of the scale-invariant $L^2$ norm of the second fundamental form. 

Further, we obtain a graph description, in the Lipschitz multi-valued sense, for any stable minimal immersion of dimension $n\geq 2$, that may have a singular set $\Sigma$ of locally finite $\mathcal{H}^{n-2}$-measure, and that is weakly close to a hyperplane. (In fact, if the $\mathcal{H}^{n-2}$-measure of the singular set vanishes, the conclusion is strengthened to a union of smooth graphs.) This follows directly from an $\eps$-regularity theorem, that assumes smallness of the scale-invariant $L^2$ tilt-excess (verified when the hypersurface is weakly close to a hyperplane). 
Specialising the multi-valued decomposition to the case of embeddings, we recover the Schoen--Simon theorem (1981). 

In both $\eps$-regularity theorems the relevant quantity (respectively, length of the second fundamental form and tilt function) solves a non-linear PDE on the immersed minimal hypersurface. The proof is carried out intrinsically (without linearising the PDE) by implementing an iteration method \`{a} la De Giorgi (from the linear De Giorgi--Nash--Moser theory). Stability implies estimates (intrinsic weak Caccioppoli inequalities) that make the iteration effective despite the non-linear framework. (In both $\eps$-regularity theorems the method gives explicit constants that quantify the required smallness.)
\end{abstract}

\section{Introduction}

\noindent{\textbf{Part \ref{partI}: curvature estimates}}. In the renowned 1975 work, Schoen--Simon--Yau proved that any properly immersed two-sided stable minimal hypersurface $M$ in $\R^{n+1}$, with $n\leq 5$, and with Euclidean mass growth at infinity, is necessarily a union of affine hyperplanes. We develop a new and alternative approach to this, that additionally solves the case $n=6$, which had since remained a well-known open question. More precisely, we prove:

\begin{thm}[Bernstein-type theorem]
\label{thm:Bernstein6}
Let $M$ be a (smooth) properly immersed two-sided stable minimal hypersurface in $\R^{n+1}$, for $n\in \{2, 3, 4, 5, 6\}$, with Euclidean mass growth at infinity, i.e.~there exists $\Lambda \in (0, \infty)$ such that $\mathcal{H}^n\big(M \cap B_R^{n+1}(0)\big)\leq \Lambda R^n$ for all $R>0$. Then $M$ is a union of affine hyperplanes.
\end{thm}

As is well-known, an equivalent formulation of this property is given via a priori (interior) curvature estimates, as follows:

\begin{thm}[Pointwise curvature estimates]
\label{thm:curv_est_6}
Assume that $M$ is a (smooth) properly immersed two-sided stable minimal hypersurface in $B_{4R}^{n+1}(0)$, with $0 \in M$ and $n\in \{2, 3, 4, 5, 6\}$, with $\frac{\mathcal{H}^n \big(M \cap B_{4R}^{n+1}(0)\big)}{(4R)^n}\leq \Lambda \in (0, \infty)$. There exists $\beta>0$ depending only on $\Lambda$ and $n$ such that
\[\sup_{x\in B_{\frac{R}{2}}^{n+1}(0)}|A_M|(x) \leq \frac{\beta}{R},\]
where $A_M$ is the second fundamental form of $M$.
\end{thm}

\begin{oss}
As usual with interior-type estimates, the choice of $\frac{1}{8}$ as ratio between the relevant radii is arbitrary. Any ratio smaller than $1$ can be allowed, and the constant $\beta$ will depend on the chosen ratio.
\end{oss}

\begin{oss}
Theorem \ref{thm:Bernstein6} holds if we replace properness with the fact that the immersion is complete. Indeed, combining the coarea formula and the isoperimetric inequality one can bound from below the $n$-area of an intrinsic ball of radius $r$ with $c(n) r^n$, for a dimensional constant $c(n)$ (see e.g.~\cite{Tysk}). This implies (using completeness and area bounds) that the immersion is proper.
\end{oss}

As recalled above, the Schoen--Simon--Yau theory obtained these theorems for $n\leq 5$, see \cite[Theorem 3]{SSY}. (In fact, under the weaker assumption of Euclidean area growth in the intrinsic sense.) For $n=6$, the validity of these properties under an additional multiplicity-$2$ condition, more precisely under the restriction $\Lambda<3\, \omega_6$, was obtained by Wickramasekera, see \cite[Theorems 9.1 and 9.2]{WicDG} (the notation $\omega_n$ stands for the $n$-volume of the $n$-dimensional unit ball).
If $M$ is assumed to be properly embedded, rather than immersed, the above results are known to be valid for $n\leq 6$ in view of the fundamental sheeting theorem by Schoen--Simon, see \cite[Theorems 1 and 3]{SchSim}. For $n \geq 7$, on the other hand, it has long been known that the situation is drastically different and the above results do not hold, even assuming that $M$ is properly embedded, as in the example of the Hardt--Simon foliation \cite{HS}.

We recall that the case $n=2$ can be treated by means of a logarithmic cut-off argument, and, in fact, it has long been known that the above theorems hold for $n=2$ without any mass hypothesis (see do Carmo--Peng \cite{dCP}, Fischer-Colbrie--Schoen \cite{FCSc}, and Pogorelov \cite{Pog}). Very recently, the Euclidean mass growth assumption has been shown to be redundant for $n=3$ in the work of Chodosh--Li \cite{ChoLi}, which resolved a long-standing conjecture of Schoen (see also subsequent alternative proofs by the same authors \cite{ChoLi2} and by Catino--Mastrolia--Roncoroni \cite{CaMaRo}). Such Bernstein theorems imply that, for $n=2,3$, the constant $\beta$ in Theorem \ref{thm:curv_est_6} does not depend on $\Lambda$.\footnote{After the appearance of this paper, Chodosh--Li--Minter--Stryker \cite{CLMS} and Mazet \cite{Maz} showed the validity of `stable Bernstein' theorems without any mass growth hypothesis, respectively for $n=4$ and $n=5$. Consequently, in these dimensions, $\beta$ in Theorem \ref{thm:curv_est_6} does not depend on $\Lambda$.}

\medskip

We obtain Theorems \ref{thm:Bernstein6} and \ref{thm:curv_est_6} as a consequence of an $\eps$-regularity theorem, in which the relevant (small) quantity is the scale-invariant $L^2$-norm of the second fundamental form:

\begin{thm}[$\eps$-regularity for the second fundamental form]
\label{thm:eps_reg_A}
Let $n\leq 6$. There exists $\epsilon_0>0$, depending only $n$ (a sufficiently small $\epsilon_0$ is explicitly given in \eqref{eq:eps_0_quant} below) with the following significance. Let $M$ be a properly immersed two-sided stable minimal hypersurface in $B_{2R}^{n+1}(0)$, with $0 \in M$ and with 
\[\frac{1}{(2R)^{n-2}}\int_{M\cap B_{2R}^{n+1}(0)} |A_M|^2 \leq \epsilon_0.\] 
Then for every $x\in M\cap B_{R/2}^{n+1}(0)$ we have $|A_M|(x)\leq \frac{1}{R}$. More precisely, in the above smallness regime, we have, for a (explicit) dimensional constant $c(n)$,
\[\sup_{M\cap B^{n+1}_{R/2}(0)}|A_M|\leq \frac{c(n)}{R} \Bigg( \frac{1}{(2R)^{n-2}}\int_{M\cap B_{2R}^{n+1}(0)} |A_M|^2 \Bigg)^{\frac{1}{2}}.\]
\end{thm}

Theorem \ref{thm:eps_reg_A} is established by PDE methods, working intrinsically on the immersed hypersurface. The relevant (non-linear) PDE is the Simons equation for $|A_M|$. The proof is obtained by implementing an iteration scheme \`{a} la De Giorgi, in the style of the linear theory in \cite{DG} (the widely known De Giorgi--Nash--Moser theory). The iteration relies on the validity of a weak intrinsic Caccioppoli inequality, valid for level set truncations of $|A_M|$. We establish this inequality in Lemma \ref{lem:Caccioppoli_A}, from the associated PDE, making (essential) use of the stability hypothesis to control the terms that escape the linear PDE theory framework. 

\medskip

Theorems \ref{thm:Bernstein6} and \ref{thm:curv_est_6} then follow employing soft classical geometric measure theory arguments: Allard's compactness, tangent cone analysis, Federer's dimension reduction. 
Ultimately as a consequence of the well-known Simons classification of stable cones (\cite{Simons}), the analysis only needs to address scenarios in which $M$ is close to a cone of the following very specific types: a single hyperplane with multiplicity, a classical cone\footnote{A classical cone is the union of three or more (distinct) closed half-hyperplanes, having for boundary a common $(n-1)$-dimensional subspace of $\R^{n+1}$, and all intersecting at said boundary, each half-hyperplane endowed with an integer multiplicity. The common $(n-1)$-dimensional subspace is also referred to as the spine, a term which in general denotes the maximal subspace along which a cone is translation invariant.}, or a union of hyperplanes with multiplicity. (Closeness is understood in the sense of varifolds.) In all these cases, the smallness condition of Theorem \ref{thm:eps_reg_A} is verified. In the first scenario, this is a consequence of the control of the tilt-excess by the height-excess (recalled in Remark \ref{oss:height_excess}) and of Schoen's inequality (\cite{Sch}, \cite{SchSim}, see also \eqref{eq:Schoen_ineq_k} below taken with $k=0$). In the remaining ones, smallness follows by an inductive argument (on the dimension of the spine of the cone), also using a higher integrability estimate close to the spine.

\begin{oss}
For $n=3$, the proof of Theorem \ref{thm:eps_reg_A} does not require any smallness assumption, thus it establishes directly Theorems \ref{thm:Bernstein6} and \ref{thm:curv_est_6} (see Appendix \ref{n=3}). In fact, this argument can be carried out also under an intrinsic area growth hypothesis.
\end{oss}

\noindent{\textbf{Part \ref{partII}: Towards a compactness theory for branched stable minimal immersions}}. In the second part of this work, we consider stable minimal immersed hypersurfaces of arbitrary dimension $n$, and allow a singular set. More precisely, we enlarge the class of two-sided stable minimal immersions by allowing $M$ to have a singular set with vanishing $2$-capacity, in particular, we allow it to have locally finite $\mathcal{H}^{n-2}$-measure. Classical branch points show that singularities of this size can indeed arise for the class in question (in contrast with the embedded case, see \cite{SchSim} and Theorem \ref{thm:SchSim} below, in which case branching can be ruled out a posteriori).

We take an intrinsic PDE approach that appears to be new in regularity theory for minimal hypersurfaces. 
The key step is an $\eps$-regularity result for the scale-invariant tilt excess, Theorem \ref{thm:eps_reg_tilt} below.
The tilt function on $M$ is 
\[g=\sqrt{1-(\nu_{M} \cdot e_{n+1})^2},\]
where $\nu_{M}$ is a choice of unit normal to $M$ and $e_{n+1}$ is (any fixed unit vector, which we can without loss of generality assume to be) the last coordinate vector. The tilt-function thus varies in $[0,1]$, with values being higher where the tangent to $M$ tilts more with respect to the reference hyperplane $\{x_{n+1}=0\}$. The smallness condition is assumed on the scale-invariant $L^2$-norm of $g$. More precisely, letting $C_{r}$ denote the cylinder $B_{r}^{n}(0) \times (-r, r)$, for $R>0$ the scale-invariant $L^2$ tilt-excess of $M$ on $C_R$ is the quantity $E_M(R)$ defined by 
\[E_M(R)^2=\frac{1}{R^n}\int_{M\cap C_R} (1-(\nu_{M} \cdot e_{n+1})^2).\]

\begin{thm}[$\eps$-regularity for the tilt]
\label{thm:eps_reg_tilt}
Let $n\geq 2$. Let $M$ be a properly immersed, two-sided, stable minimal hypersurface in $C_{2R} \setminus \Sigma$, where $\Sigma$ is closed in $C_{2R}=B^n_{2R}(0) \times (-2R, 2R)$, and with $\text{cap}_2(\Sigma)=0$ (in particular, $\mathcal{H}^{n-2}\big(\Sigma \cap K\big)<\infty$ for every $K\subset \subset C_{2R}$ is permitted). There exists a positive dimensional constant $k(n)$ (a sufficiently small $k(n)$ is given explicitly in \eqref{eq:dim_constant_1} below) with the following significance. Assume that
\[E_M(R)^2=\frac{1}{R^n}\int_{M\cap C_R} g^2 \leq k(n).\]
Then 
\[\sup_{M\cap C_{\frac{R}{2}}} g\leq \frac{1}{2n}.\]
In fact, there exists a dimensional constant $c(n)$ such that, if $E_M(R)^2 \leq k(n)$, then for every $x\in  M\cap C_{\frac{R}{2}}$ we have
\[g(x) \leq c(n) E_M(R).\]
\end{thm}

Theorem \ref{thm:eps_reg_tilt} is applicable when $M \cap C_R$ is sufficiently close to a hyperplane, which we can assume to be $\{x_{n+1}=0\}$ by a suitable rotation. This follows from the standard control of $E_R(M)$ by means of the $L^2$ height-excess (as recalled in Remark \ref{oss:height_excess} this is an easy consequence of the minimality assumption).
The bound on $g$ obtained in Theorem \ref{thm:eps_reg_tilt} forces a decomposition of $\overline{M} \cap C_{\frac{R}{2}}$ into a union of graphs (over $B^{n+1}_{\frac{R}{2}} \cap \{x_{n+1}=0\}$), where $\overline{M}$ denotes the closure of $M$ in $C_{2R}$. In the general case considered in Theorem \ref{thm:eps_reg_tilt} these graphs are Lipschitz. However, with stronger assumptions we obtain stronger conclusions as well. We illustrate three main instances: the general case of singular immersions, the case of smooth immersions, the case of singular embeddings, respectively Theorems \ref{thm:immersion_sing_sheeting}, \ref{thm:immersion_smooth_sheeting} and \ref{thm:SchSim} below. 
All three follow from the $\eps$-regularity result, Theorem \ref{thm:eps_reg_tilt}, very directly. 

\medskip
  
Theorem \ref{thm:eps_reg_tilt} is (as was the case for Theorem \ref{thm:eps_reg_A}) established by PDE methods, working intrinsically on the immersed hypersurface $M$. 
The relevant (non-linear) PDE, for the tilt-function $g$, is a direct consequence of the Jacobi field equation for $(\nu \cdot e_{n+1})$. We prove that a weak intrinsic Caccioppoli inequality is valid for level set truncations of $g$, see Lemma \ref{lem:Caccioppoli_g_sing}; this uses the associated PDE and the stability hypothesis (to control the terms that escape the linear PDE theory framework, which involve $|A_M|$). We then implement an iteration \`{a} la De Giorgi. In Remarks \ref{oss:CI_A} and \ref{oss:CI_g} we discuss similarities and differences between the two intrinsic weak Caccioppoli inequalities (Lemmas \ref{lem:Caccioppoli_A} and \ref{lem:Caccioppoli_g_sing}), as well as compare them to those in De Giorgi's work \cite{DG}.

\begin{thm}[sheeting theorem for singular immersions]
\label{thm:immersion_sing_sheeting}
In the hypotheses of Theorem \ref{thm:eps_reg_tilt}, with the further assumption that $\sup_{M \cap C_{2R}} |x_{n+1}|<\frac{R}{2}$, we have
\[\overline{M} \cap C_{\frac{R}{2}} = \cup_{j=1}^q \text{graph}(u_j)\]
for some $q\in \N$, and $u_j:B^n_{\frac{R}{2}} \to \R$ are Lipschitz functions, with Lipschitz constant at most $\frac{1}{2n}$ and $u_j \leq u_{j+1}$ for every $j \in \{1, \ldots, q-1\}$. (We identify $B^n_{\frac{R}{2}}$ with $B^{n+1}_{\frac{R}{2}}(0) \cap \{x_{n+1}=0\}$ and the target $\R$ with the $x_{n+1}$ coordinate axis.) More precisely, the Lipschitz constant of each $u_j$ is bounded by $c(n) E_M(R)$, with $c(n)$ a dimensional constant.
\end{thm}

In fact, $M \cap C_{\frac{R}{2}}$ in Theorem \ref{thm:immersion_sing_sheeting} is naturally a smooth $q$-valued function on $B^n_{\frac{R}{2}} \setminus \pi(\Sigma)$, where $\pi:\R^n \times \R \to \R^n$ is the standard projection. In general, we do not have $q$ smooth graphs on $B_{\frac{R}{2}} \setminus \pi(\Sigma)$: the size of $\Sigma$ permits classical branching, therefore smoothness only holds for the $q$-valued function and, in general, there is no ``selection'' of $q$ smooth functions on $B_{\frac{R}{2}} \setminus \pi(\Sigma)$. On the other hand, one can easily write $\overline{M} \cap C_{\frac{R}{2}}$ as the union of graphs of $q$ Lipschitz functions by ordering the $q$ values increasingly and extending the Lipschitz functions across $\pi(\Sigma)$, as done in the statement of Theorem \ref{thm:immersion_sing_sheeting}. 

The multi-valued graph structure obtained in Theorem \ref{thm:immersion_sing_sheeting} rules out, for example, that (in a branched stable minimal immersion with singular set of locally finite $(n-2)$-measure) there may be an accumulation of necks (connecting different sheets) onto a flat branch point.

\begin{oss}
Working with immersions of the same type as in Theorem \ref{thm:immersion_sing_sheeting}, and under an additional `multiplicity 2' assumption, Wickramasekera \cite{WicDG} obtained that when the $L^2$ height-excess $\int_{C_R} |x_{n+1}|^2$ is sufficiently small then a sheeting description is valid by means of a $C^{1,\alpha}$ $2$-valued function. (The strategy in \cite{WicDG} involves a $2$-valued Lipschitz approximation of $M$, followed by a linearisation of the problem, which in particular prevents a quantitative smallness condition.)
\end{oss}

Theorem \ref{thm:immersion_sing_sheeting} advances towards a compactness theory for branched stable minimal immersions. (This was obtained for a ``multiplicity $2$ class'' in \cite{WicDG}.) The natural missing step is the analysis of the situation in which, rather than being close to a hyperplane with multiplicity, $M$ is close to a classical cone. In view of Theorem \ref{thm:immersion_sing_sheeting}, and of the multiplicity-$2$ case in \cite{WicDG}, it seems natural to expect that:

\noindent \emph{Conjecture}: the class of branched two-sided stable minimal $n$-dimensional immersions with singular set of locally finite $(n-2)$-measure is compact under varifold convergence.

A further natural aim would then be to obtain a finer structure result for said singular set, in the style of \cite[Theorem 1.5]{WicDG}. It may be possible to use an intrinsic approach. This lies outside the scope of this work. 

\begin{oss}[\emph{Unique tangent hyperplanes and Bernstein-type theorem}]
An immediate byproduct of Theorem \ref{thm:immersion_sing_sheeting}, for the class of immersions under study, is that: if $x\in \overline{M}$ is such that one tangent cone (in the sense of varifolds) to $M$ at $x$ is supported on a hyperplane, then that is the unique tangent cone at $x$ (Corollary \ref{cor:uniq_tang}). Similarly, it follows immediately from Theorem \ref{thm:immersion_sing_sheeting} that, if $M$ is entire, with Euclidean mass growth, and if one tangent cone at infinity is a hyperplane with multiplicity, then $M$ is a union of hyperplanes (Corollary \ref{cor:uniq_tang_infinity}). 
\end{oss}

If the singular set $\Sigma$ in Theorem \ref{thm:eps_reg_tilt} is a priori assumed empty then the graphical decomposition is stronger and prompts a linear PDE behaviour, with a linear (interior) control of $\sup|A_M|$ by $E_M(R)$:

\begin{thm}[sheeting theorem for smooth immersions]
\label{thm:immersion_smooth_sheeting}
In the hypotheses of Theorem \ref{thm:eps_reg_tilt}, with the further assumptions that $\sup_{M \cap C_{2R}} |x_{n+1}|<\frac{R}{2}$ and that $\Sigma=\emptyset$ (that is, $M$ is a closed immersed hypersurface in $C_{2R}$), we have
\[M \cap C_{\frac{R}{2}} = \cup_{j=1}^q \text{graph}(v_j),\]
where $v_j: B^n_{\frac{R}{2}}(0) \equiv B_{\frac{R}{2}}^{n+1}(0) \cap \{x_{n+1}=0\} \to \R \equiv \text{span}(e_{n+1})$ are smooth functions and $q\in \N$. (We note that these graphs are not ordered, they may cross.) Moreover, $\sup_{j\in\{1, \ldots, q\}} \|\nabla v_j\|_{C^{1,\alpha}\big(B^n_{\frac{R}{2}}(0)\big)} \leq c(n) E_M(R)$, for a dimensional constant $c(n)$. In particular,
\[\sup_{C_{\frac{R}{2}}\cap M}|A_{M}| \leq c(n) E_M(R).\]
\end{thm}

\begin{oss}
Theorem \ref{thm:immersion_smooth_sheeting} rules out the appearance of a flat branch point when taking a (varifold) limit of smooth stable minimal immersions.
\end{oss}

\begin{oss}[\emph{Sufficiency of $\mathcal{H}^{n-2}(\Sigma)=0$}]
The assumption $\Sigma=\emptyset$ in Theorem \ref{thm:immersion_smooth_sheeting} can be weakened to $\mathcal{H}^{n-2}(\Sigma)=0$. Indeed, under this assumption, $B_{\frac{R}{2}} \setminus \pi(\Sigma)$ is a simply connected open set. This implies (using the conclusions of Theorem \ref{thm:immersion_sing_sheeting}) that $M$ can be written as the union of graphs (not ordered ones) of smooth functions on $B_{\frac{R}{2}} \setminus \pi(\Sigma)$. Then a removal of singularity for the minimal surface PDE (\cite{Sim77}) shows that in fact $\Sigma=\emptyset$. In view of this (in analogy with the earlier discussion on the branched case), a compactness theory for stable minimal immersions with a codimension-$7$ singular set could be obtained\footnote{After the appearance of this paper, Hong--Li--Wang \cite{HLW} carried out a modification of the arguments in Sections \ref{proof_eps_A} and \ref{proof_thm_A} below, obtaining \cite[Corollary 1.3]{HLW}, which asserts the validity of such compactness statement.} upon addressing the natural missing step, in which $M$ is close to a classical cone, rather than to a hyperplane with multiplicity. (Again, \cite[Theorem 1.3]{WicDG} obtained this for a ``multiplicity $2$ class''.)
\end{oss}

If we instead specialise Theorem \ref{thm:immersion_sing_sheeting} to hypersurfaces $M$ that are \emph{embedded} away from $\Sigma$, then by removal of singularities for the minimal surface PDE, used for each function $u_j$, we recover (a quantitative version of) the well-known Schoen--Simon sheeting theorem (\cite[Theorem 1]{SchSim}):

\begin{thm}[sheeting theorem for singular embeddings]
\label{thm:SchSim}
Let $n\geq 2$. Let $M$ be a properly embedded, two-sided stable minimal hypersurface in $C_{2R}  \setminus \Sigma$, where $\Sigma$ is closed in $C_{2R}=B^n_R(0) \times (-R, R)$, with locally finite $\mathcal{H}^{n-2}$-measure, or, more generally, with $\text{cap}_2(\Sigma) =0$. Assume that $\sup_{M \cap C_{2R}} |x_{n+1}|<\frac{R}{2}$ and
\[E_M(R)^2=\frac{1}{R^n}\int_{M\cap C_R} (1-\nu \cdot e_{n+1})^2 \leq k(n),\]
where $k(n)$ is the (positive) dimensional constant in \eqref{eq:dim_constant_1} and $\nu$ is a choice of unit normal to $M$. Then 
\[\overline{M} \cap C_{\frac{R}{2}}=\cup_{j=1}^q \text{graph}(u_j)\]
with $u_j:B_{\frac{R}{2}} \to \R$ smooth, $u_j < u_{j+1}$ for every $j$. In particular, $\overline{M} \cap C_{\frac{R}{2}}$ is smoothly embedded (equivalently, $M$ extends smoothly across $\Sigma$ in $C_{\frac{R}{2}}$), and $\sup_{C_{\frac{R}{2}}\cap M}|A_{M}| \leq c(n) E_M(R)$ for a dimensional constant $c(n)$. 
\end{thm}
We recall that Theorem \ref{thm:SchSim} leads (by fairly standard arguments) to the renowned compactness and regularity theory \cite[Theorems 2 and 3]{SchSim} for stable minimal embedded hypersurfaces that are allowed to possess a singular set of locally finite $\mathcal{H}^{n-2}$-measure. A posteriori, the singular set contains no branch points and in fact has dimension at most $n-7$, it is discrete in the case $n=7$, and empty for $n\leq 6$.

We thus obtain an alternative and hopefully more immediate route to (the main component of) the Schoen--Simon theory. (The approach in \cite{SchSim} involves a partial $q$-valued graph decomposition of the embedding, and a linearisation of the problem, both of which we avoid.)  

\medskip

The impact of Schoen--Simon's compactness, and of Schoen--Simon--Yau's curvature estimates, for developments in analysis and geometry over the last half century, cannot be overstated.

\part{Curvature estimates}
\label{partI}

We will denote by $M$ (and by $M_\ell$, $\ell \in \N$, when considering a sequence) a smooth two-sided properly immersed stable minimal hypersurface in an open set $U\subset \R^{n+1}$. Typically, the open set $U$ will be a ball $B_R^{n+1}(0)$, or the whole of $\R^{n+1}$, or a cylinder of the form $B_R^{n}(0) \times (-R, R)$. In other words $M=\iota(S)$, with $S$ an $n$-dimensional manifold and $\iota:S\to U$ a proper two-sided stable minimal immersion. We recall that the stability condition is the non-negativity of the second variation of the $n$-area, and that this amounts to the validity of 
\[\int_S |A_M|^2 \phi^2 \leq \int_S |\nabla \phi|^2\]
for any $\phi \in C^1_c(S)$, where $S$ is endowed with the pull-back metric from $U$, $\nabla$ is the metric gradient on $S$ and $|A_M|$ the length of the second fundamental form. We note that whenever $\varphi \in C^1_c(U)$, then $\varphi\circ \iota\in C^1_c(S)$ (since the immersion is proper); with a slight abuse of notation, we will write the integrals directly on $M$, with $\int_M |\nabla \varphi|^2$ in place of $\int_S |\nabla (\varphi\circ \iota)|^2$ and the inequality taking the form $\int_M |A_M|^2 \varphi^2 \leq \int_M |\nabla \varphi|^2$.

\section{Proof of Theorem \ref{thm:eps_reg_A}}
\label{proof_eps_A}

In this section we prove Theorem \ref{thm:eps_reg_A}. We recall the well-known Simons identity (\cite{Simons}), for the second fundamental form $A$ of a minimal hypersurface:

\begin{equation}
 \label{eq:Simons}
\frac{1}{2} \Delta |A|^2 = |\nabla A|^2 - |A|^4.
\end{equation}
Clearly, $|\nabla A| \geq \big|\nabla|A|\big|$; in \cite[(1.33)]{SSY} it is shown that the minimality condition implies the following improved inequality, with $c=\frac{2}{n}$:
\begin{equation}
 \label{eq:SSY_ineq}
|\nabla A|^2 \geq (1+c) \big|\nabla |A|\big|^2.
\end{equation}
We will also use the following variant of (\ref{eq:Simons}): as $\frac{1}{2}\Delta |A|^2 = |A| \Delta |A| + |\nabla |A||^2$, we find 
\begin{equation}
 \label{eq:Simons2}
|A| \Delta |A| = |\nabla A|^2- \big|\nabla |A|\big|^2 - |A|^4.
\end{equation}

The following lemma contains the relevant weak (intrinsic) Caccioppoli inequality for the level set truncations of $|A|$:

\begin{lem}
 \label{lem:Caccioppoli_A}
Let $M$ be a properly immersed smooth two-sided stable minimal hypersurface in $U\subset \R^{n+1}$. For any $k\geq 0$ and any $\eta \in C^1_c(U)$ we have
\begin{equation*}
 \begin{aligned}
\int_{\{|A|>k\}} \Bigg(1-\frac{k}{|A|}\Bigg)\,\big|\nabla |A|\big|^2 \, \eta^2 \leq  \frac{1}{c}  \int ((|A|-k)^+)^2 |\nabla \eta|^2 +\\
\frac{k}{c} \int ((|A|-k)^+)^3\, \eta^2 + \frac{2 k^2}{c} \int ((|A|-k)^+)^2\, \eta^2 + \frac{k^3}{c} \int (|A|-k)^+\, \eta^2.
 \end{aligned}
\end{equation*}
\end{lem}

\begin{oss}
As mentioned above, the integrals are implicitly understood to be on $S$, with $\eta \circ \iota$ in place of $\eta$. For $k=0$ the inequality is $\int  \big|\nabla |A|\big|^2 \, \eta^2 \leq  \frac{1}{c}  \int |A|^2 |\nabla \eta|^2$, which appears as an intermediate step along the proof of \cite[Theorem 1]{SSY}.
\end{oss}

\begin{proof}
We use the stability inequality with the Lipschitz test function $(|A|-k)^+ \eta$, for $k \in [0, \infty)$ and $\eta \in C^1_c(U)$. We note that ${(|A|-k)^+}^2 \in C^1(M) \cap W^{2,\infty}_{\text{loc}}(M)$: indeed, being Lipschitz, its (distributional) gradient is the function $\nabla {(|A|-k)^+}^2 = 2 {(|A|-k)^+} \nabla {(|A|-k)^+}=2 {(|A|-k)^+} \nabla |A|$, which in turn is locally Lipschitz. We find

\begin{equation*}
 \begin{aligned}
\int |A|^2 (|A|-k)^+)^2 \eta^2 \leq \int |\nabla ((|A|-k)^+ \eta)|^2 = \\
\int |\nabla (|A|-k)^+|^2 \eta^2 + 2 \int (|A|-k)^+ \,\eta \,\nabla (|A|-k)^+ \, \nabla \eta + \int  ((|A|-k)^+)^2 |\nabla \eta|^2 \\
=\int |\nabla (|A|-k)^+|^2 \eta^2 + \frac{1}{2} \underbrace{\int \nabla ((|A|-k)^+)^2 \, \nabla \eta^2}_{} + \int  ((|A|-k)^+)^2 |\nabla \eta|^2 
\end{aligned}
\end{equation*}
and integrating by parts the braced term we can continue the equality chain
\begin{equation*}
 \begin{aligned}
=\int |\nabla (|A|-k)^+|^2 \eta^2 - \frac{1}{2} \int_{\{|A|>k\}} \Delta |A|^2 \, \eta^2 + \int_{\{|A|>k\}} \frac{k}{|A|} |A| \Delta |A| \, \eta^2 \\+ \int  ((|A|-k)^+)^2 |\nabla \eta|^2 
\end{aligned}
\end{equation*}
\begin{equation*}
 \begin{aligned}
\underbrace{=}_{(\ref{eq:Simons}), (\ref{eq:Simons2})}\int_{\{|A|>k\}} \big|\nabla |A|\big|^2 \eta^2 + \int_{\{|A|>k\}}  (-|\nabla A|^2 + |A|^4) \,\eta^2 +\\  \int_{\{|A|>k\}} \frac{k}{|A|}\big( |\nabla A|^2- \big|\nabla |A|\big|^2\big) \, \eta^2 - k\int_{\{|A|>k\}} |A|^3 \, \eta^2  + \int  ((|A|-k)^+)^2 |\nabla \eta|^2.
 \end{aligned}
\end{equation*}
The left-most side (of the above chain of inequalities) is expanded as follows: $\int |A|^2 (|A|-k)^+)^2 \eta^2= \int_{\{|A|>k\}} (|A|^4 -2k|A|^3 + k^2 |A|^2) \eta^2$. We thus find
\begin{equation*}
 \begin{aligned}\int_{\{|A|>k\}} \left(1-\frac{k}{|A|}\right)(|\nabla A|^2- \big|\nabla |A|\big|^2)\eta^2 \leq \\
\int  ((|A|-k)^+)^2 |\nabla \eta|^2 + k \int_{\{|A|>k\}} |A|^3 \, \eta^2 -  k^2 \int_{\{|A|>k\}} |A|^2 \, \eta^2 
\end{aligned}
\end{equation*}
and using (\ref{eq:SSY_ineq}) we conclude
\begin{equation}
 \label{eq:Caccioppoli_A}
\int_{\{|A|>k\}} \Bigg(1-\frac{k}{|A|}\Bigg)\,\big|\nabla |A|\big|^2 \, \eta^2 \leq  
\end{equation}
\[\frac{1}{c}  \int \big((|A|-k)^+\big)^2 |\nabla \eta|^2 + \frac{k}{c} \int  |A|^2 (|A|-k)^+\, \eta^2 .\]
The desired inequality follows upon rewriting the last term on the right-hand-side of (\ref{eq:Caccioppoli_A}), on the set ${\{|A|>k\}}$:
\begin{equation}
|A|^2(|A|-k) = \left(|A|-k+k\right)^2 (|A|-k) =
(|A|-k)^3 +2k(|A| -k)^2 + k^2(|A|-k).
\end{equation}
\end{proof}

Lemma \ref{lem:Caccioppoli_A} provides the weak intrinsic Caccioppoli inequality that will lead, through an iteration scheme \`{a} la De Giorgi, to Theorem \ref{thm:eps_reg_A}.

\begin{oss}
Note that both assumption and conclusion in Theorem \ref{thm:eps_reg_A} are scale-invariant. The scale-invariant quantity $\frac{1}{(R)^{n-2}}\int_{M\cap B_{R}(0)} |A|^2$ is uniformly bounded under the Euclidean mass growth hypothesis that we have. Indeed, stability used with a test function $\varphi \in C^1_c(B_{2R}(0))$ with $\varphi \equiv 1$ on $B_R(0)$ and $|\nabla \varphi|\leq \frac{2}{R}$ gives $\int_{B_R(0)}|A|^2 \leq \frac{4}{R^2}\Big(\int_{B_{2R}(0)}1\Big)\leq 2^{n+2} R^{n-2}\Lambda$.
\end{oss}

\begin{oss}
\label{oss:catenoid}
The standard catenoids for $n\geq 3$, for which $\frac{1}{R^{n-2}}\int_{B_R(0)\cap M} |A|^2 \to 0$ as $R\to \infty$, show that stability is essential for Theorem \ref{thm:eps_reg_A}. (On the other hand, the dimensional restriction may not be essential.) 
Catenoids also show that, given $M$ minimal, the relevant ``scale-invariant energy'' $\frac{1}{R^{n-2}}\int_{B_R(0)\cap M} |A|^2$ is not an increasing function of $R$ (unlike in several well-known $\eps$-regularity theorems).
\end{oss}

\begin{proof}[Proof of Theorem \ref{thm:eps_reg_A}]
We will assume $n\geq 3$ (see Remark \ref{oss:A_n2}) and consider the sequences $k_\ell = d-\frac{d}{2^{\ell-1}}$ and $R_\ell = \frac{R}{2}+\frac{R}{2^\ell}$ for $\ell \in \{1, 2, \ldots,\}$ (respectively increasing and decreasing), where $d>0$ is for the moment left undetermined, and will be quantified in terms of $\frac{1}{(2R)^{n-2}}\int_{B_{2R}} |A|^2$. Here and below, when writing integrals on $B_r=B^{n+1}_r(0)$ we implicitly understand that the integration is on $\iota^{-1}(B^{n+1}_r(0))$.

\medskip

Using Lemma \ref{lem:Caccioppoli_A} with $k_{\ell}$ in place of $k$, noting the inclusion $\{|A|>k_{\ell}\} \supset \{|A|>k_{\ell+1}\}$, 
and that on the set $\{|A|>k_{\ell+1}\}$ we have $\left(1-\frac{k_{\ell}}{|A|}\right)^+ \geq  1-\frac{k_{\ell}}{k_{\ell+1}} = \frac{d}{k_{\ell+1} 2^{\ell}} \geq \frac{1}{2^{\ell}}$ we have
\begin{equation*}
 \begin{aligned}
\frac{1}{2^\ell} \int_{\{|A|>k_{\ell+1}\}}  |\nabla |A||^2 \, \eta^2 \leq \frac{1}{c}  \int ((|A|-k_{\ell})^+)^2 |\nabla \eta|^2 + 
 \frac{k_{\ell}}{c} \int ((|A|-k_{\ell})^+)^3\, \eta^2 \\
 +\frac{2 k_{\ell}^2}{c} \int ((|A|-k_{\ell})^+)^2\, \eta^2 + \frac{k_{\ell}^3}{c} \int (|A|-k_{\ell})^+\, \eta^2.
 \end{aligned}
\end{equation*}
Since 
\[|\nabla ((|A|-k_{\ell+1})^+ \eta)|^2 \leq  2\chi_{\{|A|>k_{\ell+1}\}} |\nabla |A||^2 \, \eta^2 + 2((|A|-k_{\ell+1})^+)^2 |\nabla \eta|^2,\]
and $(|A|-k_{\ell+1})^+ \leq (|A|-k_{\ell})^+$, it follows that
\begin{equation}
\label{eq:first_A}
\begin{aligned}
\int  |\nabla ((|A|-k_{\ell+1})^+ \eta)|^2  \leq \left( \frac{2^{\ell+1}}{c}+2\right)  \int ((|A|-k_{\ell})^+)^2 |\nabla \eta|^2 + \\
\frac{k_{\ell}2^{\ell+1}}{c} \int ((|A|-k_{\ell})^+)^3\, \eta^2 + \frac{2 k_{\ell}^2 2^{\ell+1}}{c} \int ((|A|-k_{\ell})^+)^2\, \eta^2 + \\
\frac{k_{\ell}^3 2^{\ell+1}}{c} \int (|A|-k_{\ell})^+\, \eta^2.
\end{aligned}
\end{equation}
We will use the following Michael--Simon inequality (\cite[Theorem 2.1]{MS}, see also \cite{Bre}), to bound from below the left-hand-side of (\ref{eq:first_A}):
\begin{equation}
\label{eq:MichaelSimon_A}
\left(\int \left|(|A|-k_{\ell+1})^+ \eta\right|^{\frac{2n}{n-2}}\right)^{\frac{n-2}{n}} \leq  C_{MS}\int |\nabla ((|A|-k_{\ell+1})^+ \eta)|^2,
\end{equation}
for a dimensional constant $C_{MS}$ explicitly given by $C_{MS} = \left(\frac{2(n-1) 4^{n+1}}{(n-2)\omega_n^{1/n}}\right)^2$, where $\omega_n$ is the $n$-volume of the unit ball in $\R^n$.

We define, for each $\ell$, the function $\eta_\ell$ to be identically $1$ on $B_{R_{\ell+1}}$, with $\text{spt}\eta_\ell = B_{R_{\ell}}$ and $|\nabla \eta_\ell|\leq  \frac{2}{|R_{\ell}-R_{\ell+1}|} = \frac{ 2^{\ell+2}}{R}$, and with $0\leq \eta_\ell \leq 1$. From (\ref{eq:first_A}) and (\ref{eq:MichaelSimon_A}), making the choice $\eta=\eta_\ell$, we find (using $k_\ell \leq d$ on the right-hand-side of (\ref{eq:first_A}))
\begin{equation}
\label{eq:second_A}
\begin{aligned}
\frac{1}{C_{MS}}\left(\int_{B_{R_{\ell+1}}} ((|A|-k_{\ell+1})^+ )^{\frac{2n}{n-2}}\right)^{\frac{n-2}{n}} \leq \\
\left( \frac{2^{\ell+1}}{c}+2\right) \left(\frac{4^{\ell+2}}{R^2}\right) \int_{B_{R_{\ell}}} ((|A|-k_{\ell})^+)^2  + \frac{d 2^{\ell+1}}{c} \int_{B_{R_{\ell}}} ((|A|-k_{\ell})^+)^3 \\
+\frac{2 d^2 2^{\ell+1}}{c} \int_{B_{R_{\ell}}} ((|A|-k_{\ell})^+)^2  +\frac{d^3 2^{\ell+1}}{c} \int_{B_{R_{\ell}}} (|A|-k_{\ell})^+ .
\end{aligned}
\end{equation}
For the right-hand-side of (\ref{eq:second_A}) we use H\"older's inequality three times, first with $1 = \frac{n-2}{n} + \frac{2}{n}$,
\[\int_{B_{R_{\ell}}} ((|A|-k_\ell)^+ )^{2} \leq \left(\int_{B_{R_{\ell}}} ((|A|-k_\ell)^+ )^{\frac{2n}{n-2}}\right)^{\frac{n-2}{n}}  \left(\int \chi_{\{|A|>k_\ell\}\cap B_{R_{\ell}}}\right)^{\frac{2}{n}},\]
then with $\frac{n-2}{2n}+ \frac{n+2}{2n}=1$,
\[\int_{B_{R_{\ell}}} (|A|-k_\ell)^+  \leq \left(\int_{B_{R_{\ell}}} ((|A|-k_\ell)^+ )^{\frac{2n}{n-2}}\right)^{\frac{n-2}{2n}}  \left(\int \chi_{\{|A|>k_\ell\}\cap B_{R_{\ell}}}\right)^{\frac{n+2}{2n}},\]
and finally (this is possible for $n\leq 6$) with exponents $\frac{2n}{3(n-2)}$ and $\frac{2n}{6-n}$ (with $1= \frac{3n-6}{2n} + \frac{6-n}{2n}$), where in the case $n=6$ the two exponents are just $1$ and $\infty$ (hence for $n=6$ the second factor on the right--hand-side in the following inequality is just $1$)
\[\int_{B_{R_{\ell}}} ((|A|-k_\ell)^+)^3  \leq \left(\int_{B_{R_{\ell}}} ((|A|-k_\ell)^+ )^{\frac{2n}{n-2}}\right)^{\frac{3(n-2)}{2n}}  \left(\int \chi_{\{|A|>k_\ell\}\cap B_{R_{\ell}}}\right)^{\frac{6-n}{2n}}.\]
We will use the notation
\[S_\ell=\int_{B_{R_\ell}} ((|A|-k_\ell)^+ )^{\frac{2n}{n-2}},\]
and aim for a superlinear decay estimate for this quantity as $\ell \to \infty$.

We note that $\mathcal{H}^n\left(\{|A|>k_\ell\} \cap B_{R_{\ell}}\right)$
can be bounded as follows, for $\ell \geq 2$, using Markov's inequality and the fact that on the set $\{|A|>k_\ell\}$ we have $(|A|-k_{\ell-1})^+ \geq k_\ell -k_{\ell-1} = \frac{d}{2^{\ell-1}}$: 

\[\{|A|>k_\ell\} \cap B_{R_{\ell}}\subset \left\{((|A|-k_{\ell-1})^+)^{\frac{2n}{n-2}} \geq \frac{d^{\frac{2n}{n-2}}}{(2^{\frac{2n}{n-2}})^{\ell-1}}\right\} \cap B_{R_{\ell}} \Rightarrow \]
\begin{equation}
\label{eq:Markov}\mathcal{H}^n\left(\{|A|>k_\ell\} \cap B_{R_{{\ell}}}\right) \leq \frac{(2^{\frac{2n}{n-2}})^{\ell-1}}{d^{\frac{2n}{n-2}}}  \int_{B_{R_{\ell}}} ((|A|-k_{\ell-1})^+)^{\frac{2n}{n-2}} \leq   \frac{(2^{\frac{2n}{n-2}})^{\ell-1}}{d^{\frac{2n}{n-2}}}  S_{\ell-1}.
\end{equation}
The above three instances of H\"older's inequality, together with (\ref{eq:Markov}), give
\[\int_{B_{R_{\ell}}} ((|A|-k_\ell)^+ )^{2} \leq S_\ell^{\frac{n-2}{n}} \left(\frac{(2^{\frac{2n}{n-2}})^{\ell-1}}{d^{\frac{2n}{n-2}}}  S_{\ell-1} \right)^{\frac{2}{n}},\]
\[\int_{B_{R_{\ell}}} (|A|-k_\ell)^+  \leq S_\ell^{\frac{n-2}{2n}}  \left(\frac{(2^{\frac{2n}{n-2}})^{\ell-1}}{d^{\frac{2n}{n-2}}}  S_{\ell-1} \right)^{\frac{n+2}{2n}},\]
\[\int_{B_{R_{\ell}}} ((|A|-k_\ell)^+)^3  \leq S_\ell^{\frac{3(n-2)}{2n}}  \left(\frac{(2^{\frac{2n}{n-2}})^{\ell-1}}{d^{\frac{2n}{n-2}}}  S_{\ell-1} \right)^{\frac{6-n}{2n}}.\]

With these we bound from above the right-hand-side of (\ref{eq:second_A}) and obtain 

\[\frac{1}{C_{MS}} S_{\ell+1}^{\frac{n-2}{n}} \leq \left(\left( \frac{2^{\ell+1}}{c}+2\right) \left(\frac{4^{\ell+2}}{R^2}\right) + \frac{2 d^2 2^{\ell+1}}{c}\right) S_{\ell}^{\frac{n-2}{n}} \left(\frac{(2^{\frac{4}{n-2}})^{\ell-1}}{d^{\frac{4}{n-2}}}\right)  S_{\ell-1}^{\frac{2}{n}}  \]

\[+\frac{d 2^{\ell+1}}{c} S_{\ell}^{\frac{3n-6}{2n}} \left(\frac{(2^{\frac{6-n}{n-2}})^{\ell-1}}{d^{\frac{6-n}{n-2}}} \right)   S_{\ell-1}^{\frac{6-n}{2n}} + \frac{d^3 2^{\ell+1}}{c}  S_{\ell}^{\frac{n-2}{2n}} \left(\frac{(2^{\frac{n+2}{n-2}})^{\ell-1}}{d^{\frac{n+2}{n-2}}} \right)  S_{\ell-1}^{\frac{n+2}{2n}} ,\]
for every $\ell \geq 2$. Noting that $S_{\ell}\leq S_{\ell-1}$ by definition, and that $2-\frac{4}{n-2} = 1-\frac{6-n}{n-2}=3-\frac{n+2}{n-2} = \frac{2(n-4)}{n-2}$, the last inequality implies (using $n=\frac{2}{c}$)
\[S_{\ell+1}^{\frac{n-2}{n}} \leq (16n C_{MS})\, C^\ell \left(\frac{1}{R^2 d^{\frac{4}{n-2}}} + d^{\frac{2(n-4)}{n-2}}\right) S_{\ell-1},\]
where $C$ is a dimensional constant that can be explicitly taken (using rough estimates for the constants appearing above) to be $C=2^{\frac{3n-2}{n-2}}$.

Letting $d=\frac{1}{R}$ we find $S_{\ell+1}^{\frac{n-2}{n}} \leq 32n C_{MS} \, C^\ell  \frac{1}{R^{\frac{2(n-4)}{n-2}}} S_{\ell-1}$ and hence arrive at the decay relation ($\ell \geq 2$)
\[S_{\ell+1} \leq (32n C_{MS})^{\frac{n}{n-2}} \frac{\tilde{C}^\ell}{R^{\frac{2n(n-4)}{(n-2)^2}}} S_{\ell-1}^{1+\frac{2}{n-2}},\]
where the dimensional constant $\tilde{C}=C^{\frac{n}{n-2}}$ can be explicitly taken to be $\tilde{C}=2^{\frac{n(3n-2)}{(n-2)^2}}$.
This relation forces $S_{\ell} \to 0$, as long as $S_1$ is sufficiently small, by Lemma \ref{lem:analysis1}; precisely, if $S_1 \leq \frac{R^{\frac{n(n-4)}{(n-2)}}}{(32n C_{MS})^{\frac{n}{2}}\tilde{C}^{\frac{n(n-2)}{2}}}$. (We use the relation above with $\ell$ even; the sequence $S_{\ell}$ is decreasing by definition, hence it suffices to prove the convergence to $0$ for the subsequence of odd indices. Setting $2j=\ell$ and $T_j=S_{2j-1}$ for $j\geq 1$, we obtain the recursive relation $T_{j+1} \leq \frac{(32n C_{MS})^{\frac{n}{n-2}}}{R^{\frac{2n(n-4)}{(n-2)^2}}} \tilde{C}^{2j} T_{j}^{1+\frac{2}{n-2}}$ and Lemma \ref{lem:analysis1} gives $T_j \to 0$ if $T_1 =S_1$ is as specified above.)
The smallness assumption on $S_1$ can be written as follows, for the associated scale-invariant quantity (recall that $R_1=R$ and $k_1=0$ in the definition of $S_1$):
\begin{equation}
\label{eq:smallness_S_1}
 \frac{1}{R^{n-\frac{2n}{n-2}}} \int_{B_R} |A|^{\frac{2n}{n-2}}\leq \frac{1}{(32n C_{MS})^{\frac{n}{2}}\tilde{C}^{\frac{n(n-2)}{2}}}=\frac{1}{(32n C_{MS})^{\frac{n}{2}}\,2^{\frac{n^2(3n-2)}{2(n-2)}}}.
\end{equation}
Under the condition \eqref{eq:smallness_S_1}, the convergence $S_\ell \to 0$ obtained implies that $|A|\leq \frac{1}{R}$ a.e.~on $M\cap B_{\frac{R}{2}}$ (by our choice $d=1$), and thus everywhere (by smoothness of $M$). 

We next check that there exists $\epsilon_0>0$ such that \eqref{eq:smallness_S_1} is implied the hypotheses of Theorem \ref{thm:eps_reg_A}. Lemma \ref{lem:Caccioppoli_A}, used with $k=0$, gives $\int_{B_{2R}}|\nabla (|A|\psi)|^2 \leq \big(\frac{2}{c}+2\big) \int_{B_{2R}} |A|^2 |\nabla \psi|^2$ for $\psi \in C^1_c(B_{2R})$. Combining this with the Michael-Sobolev inequality, which gives $\left(\int_{B_{2R}} (|A|\psi)^{\frac{2n}{n-2}}\right)^{\frac{n-2}{n}} \leq C_{MS}  \int_{B_{2R}}|\nabla (|A|\psi)|^2$, choosing $\psi$ to be identically $1$ in $B_R$ and with $|\nabla \psi|\leq \frac{2}{R}$, we obtain (using $c=\frac{2}{n}$)
\begin{equation}
 \label{eq:higher_integrability}
\int_{B_R} |A|^{\frac{2n}{n-2}} \leq C_{MS}^{\frac{n}{n-2}}\left(\frac{4(n+2)}{R^2}\int_{B_{2R}}|A|^2\right)^{\frac{n}{n-2}} ,
\end{equation}
that is, the following inequality holds between the two relevant scale-invariant quantities:
\[\frac{1}{R^{n-\frac{2n}{n-2}}} \int_{B_R} |A|^{\frac{2n}{n-2}} \leq  (4(n+2)\, 2^{n-2}  C_{MS})^{\frac{n}{n-2}}  \left(\frac{1}{(2R)^{n-2}} \int_{B_{2R}}|A|^2\right)^{\frac{n}{n-2}}.\]

We have thus established the first conclusion of Theorem \ref{thm:eps_reg_A}, that is, there exists a dimensional $\epsilon_0>0$ for which $|A|\leq \frac{1}{R}$ on $M\cap B_{\frac{R}{2}}$. Explicitly, this can be quantified from the above by requiring $\Big(\epsilon_0\, 4(n+2)\, 2^{n-2}  C_{MS}\Big)^{\frac{n}{n-2}}\leq \frac{1}{(32n C_{MS})^{\frac{n}{2}}\,2^{\frac{n^2(3n-2)}{2(n-2)}}}$, that is, one can take
\begin{equation}
\label{eq:eps_0_quant}
\epsilon_0= \frac{1}{C_{MS}^{1+\frac{(n-2)}{2}} \, 4(n+2)\, 2^{n-2}  \,(32n)^{\frac{(n-2)}{2}}\,2^{\frac{n(3n-2)}{2}}}.
\end{equation}

To see the second assertion of Theorem \ref{thm:eps_reg_A}, we exploit the freedom on $d$.  
We choose $d=\frac{1}{mR}$ for $m \in [1,\infty)$, and the conclusion $S_\ell \to 0$, i.e.~$|A|\leq \frac{1}{mR}$ on $B_{\frac{R}{2}}$, follows if $S_1$ is suitably small. Indeed, the decay relation becomes (for $\ell \geq 2$ and for an explicit dimensional constant $C$)
\[S_{\ell+1} \leq \Bigg( m^{\frac{4}{n-2}}+\frac{1}{m^{\frac{2(n-4)}{n-2}}} \Bigg)^{\frac{n}{n-2}}\frac{C^\ell}{R^{\frac{2n(n-4)}{(n-2)^2}}} S_{\ell-1}^{1+\frac{2}{n-2}} \leq 2^{\frac{n}{n-2}} m^{\frac{4n}{(n-2)^2}}\frac{C^\ell}{R^{\frac{2n(n-4)}{(n-2)^2}}} S_{\ell-1}^{1+\frac{2}{n-2}}.\]
Hence if $S_1 \leq\frac{R^{\frac{n(n-4)}{(n-2)}}}{{m^{\frac{2n}{(n-2)}}}\overline{C}}$ (for an explicit dimensional constant $\overline{C}$) then $S_\ell \to 0$, that is, $|A|\leq \frac{1}{mR}$ on $M\cap B_{\frac{R}{2}}$. With the same considerations given for the case $m=1$ above, we have that the smallness requirement on $S_1$ is implied by our hypotheses, as long as $\frac{1}{(2R)^2}\int_{B_{2R}}|A|^2$ is sufficiently small. The explicit relations that we have obtained show, in fact, that, for a (explicit) dimensional constant $c(n)$,
\[R \sup_{M\cap B_{\frac{R}{2}}} |A| \leq c(n) \Big( \int_{B_{2R}} |A|^2\Big)^{\frac{1}{2}}.\] 
\end{proof}

\begin{oss}
\label{oss:CI_A}
De Giorgi \cite{DG} exploits the Caccioppoli inequality $\int_{\{u>k\} \cap B^n_{\rho}(p)} |D u|^2  \leq \frac{c(n)}{(r-\rho)^2} \int_{\{u>k\} \cap B^n_r(p)} {(u-k)^+}^2$ (for any $k$, with $\rho<r<R$) to prove his theorem, for a weak solution $u:B^n_R(p)\to \R$ of a linear PDE in divergence form with $L^\infty$ strictly elliptic coefficients. In Lemma \ref{lem:Caccioppoli_A} we have an \emph{intrinsic} Caccioppoli inequality on $M$, when $k=0$. The multiplicative factor $\big(1-\frac{k}{|A|} \big)^+$ appears on the left-hand-side for $k>0$: as shown, this does not disturb the iterative scheme. Extra terms (that weaken the inequality further, when comparing to the classical case) appear on the right-hand-side when $k>0$. These terms involve $L^p$-norms of the truncations up to $p=3$, with multiplicative factors $k^{4-p}$, which influence the dependence on $d$ in the decay relation. The smallness requirement in Theorem \ref{thm:eps_reg_A} is due\footnote{As observed explicitly in Appendix \ref{n=3}, smallness is only needed for $n \in \{4, 5, 6\}$.} to these extra terms (which also force the dimensional bound $n\leq 6$).

While in \cite{DG} the classical Caccioppoli inequality for $(u-k)^+$ follows from the linearity of the PDE, in our case stability provides sufficient control on the non-linearity of (\ref{eq:Simons}), leading to the weak intrinsic Caccioppoli inequality. (In the absence of stability Theorem \ref{thm:eps_reg_A} fails, as pointed out in Remark \ref{oss:catenoid}.)
\end{oss} 
 
\begin{oss}[$n=2$]
\label{oss:A_n2}
We proved Theorem \ref{thm:eps_reg_A} for $n \neq 2$. The proof adapts to treat the case $n=2$ by using the Michael--Simon inequality to embed $L^3$ into $W^{1,\frac{6}{5}}$ and the H\"older inequality to bound $\int |\nabla \big((|A|-{k_\ell})^+ \eta_\ell \big)|^{\frac{6}{5}}$ by means of the product $\Big(\int |\nabla \big((|A|-{k_\ell})^+ \eta_\ell \big)|^{2}\Big)^{\frac{3}{5}} \mathcal{H}^2\Big(M \cap B_{R_\ell} \Big)^{\frac{2}{5}}$. We do not carry out the iteration explicitly, also in view of the fact that (as mentioned in the introduction) curvature estimates for $n=2$ admit a simple treatment.
\end{oss}

\section{Proof of Theorems \ref{thm:Bernstein6} and \ref{thm:curv_est_6}}
\label{proof_thm_A}

The proof of Theorem \ref{thm:curv_est_6} is reduced to the analysis of scenarios in which Theorem \ref{thm:eps_reg_A} is applicable. The first scenario handles the case in which $M$ (properly immersed smooth two-sided and stable) is weakly close (i.e.~as a varifold) to a hyperplane with multiplicity; the second one handles the case in which $M$ is weakly close to a classical cone; the third one handles the case in which $M$ is weakly close to a (finite) union of hyperplanes with multiplicity. In all scenarios the conclusion is that $M$ must be a smooth perturbation (as an immersion) of the cone in question (which is a posteriori always a union of hyperplanes with multiplicity).

\subsection{Closeness to a hyperplane}
\label{scenario1}

Let $M$ be a properly immersed two-sided stable minimal hypersurface that is weakly close to a hyperplane with multiplicity (as varifolds). Upon rotating coordinates, we assume that the hyperplane is $\{x_{n+1}=0\}$. We recall the following standard inequality (implied by minimality, via the first variation formula, with a suitable choice of test function, see \cite[Section 22]{SimonNotes} and Remark \ref{oss:height_excess} below)
\[\Big(\frac{1}{3R}\Big)^n\int_{B_{3R}^{n+1} \cap M} |\nabla x_{n+1}|^2  \leq C(n)\Big(\frac{2}{7R}\Big)^{n+2}\int_{B_{\frac{7R}{2}}^{n+1} \cap M} |x_{n+1}|^2 ,\]
and $|\nabla x_{n+1}| = |\text{proj}_{TM}(e_{n+1})| = \sqrt{1-(\nu\cdot e_{n+1})^2}$. This says that the $L^2$ height-excess controls the $L^2$ tilt-excess linearly (both excesses defined in a scale-invariant fashion). Moreover, Schoen's inequality (\cite{Sch}, \cite[Lemma 1]{SchSim}, see also \eqref{eq:Schoen_ineq_k} below taken with $k=0$) gives
\[\frac{1}{(2R)^{n-2}}\int_{B_{2R}^{n+1} \cap M} |A|^2   \leq  2n \Big(\frac{1}{3R}\Big)^n\int_{B_{3R}^{n+1} \cap M} |\nabla x_{n+1}|^2,\]
where the scale-invariant tilt excess appears on the right-hand-side. Here and below, domains of the type $D \cap M$ for $D$ open, are implicitly understood to be the inverse image of $D$ via the immersion that gives $M$ (we will use this with $M$ belonging to a sequence of immersed hypersurfaces).

When $\hat{M}_j$ converge (as varifolds for $j\to \infty$) in $B_{4R}^{n+1}$ to $\{x_{n+1}=0\}$ with multiplicity, by the monotonicity formula we have that $\sup_{\hat{M}_j \cap B_{\frac{7R}{2}}^{n+1}}|x_{n+1}| \to 0$, so we conclude that for all sufficiently large $j$ the quantity $\frac{1}{(2R)^{n-2}}\int_{B_{2R}^{n+1} \cap \hat{M}_j} |A_{\hat{M}_j}|^2$ is at most $\epsilon_0$ and Theorem \ref{thm:eps_reg_A} applies. As immediate consequence, we have:

\begin{lem}
\label{lem:A_sheeting_1}
Let $\hat{M}_j$ be a sequence of smooth properly immersed two-sided stable minimal hypersurfaces in $B^{n+1}_{4R}(0)$, $n\leq 6$, such that $\hat{M}_j$ converge (as varifolds) to $q |P|$ as $j \to \infty$, where $P$ is a hyperplane and $q\in \N$. Then $\sup_{B^{n+1}_{\frac{R}{2}}(0) \cap \hat{M}_j} |A_{\hat{M}_j}|  \to 0$ as $j \to \infty$, and $\hat{M}_j$ converge smoothly to $P$ in $B^{n+1}_{\frac{R}{2}}(0)$ (as immersions, with the limit being an immersion with $q$ connected components, each covering $P$ once).
\end{lem}

\subsection{Closeness to a classical cone}
\label{scenario_class_cone}

Let $C$ be a classical cone, i.e.~a sum of half-hyperplanes all intersecting at a given $(n-1)$-dimensional subspace, $C=\sum_{i=1}^N q_i | H_i |$ with $q_i \in \N$, $H_i$ a half-hyperplane whose boundary is the given $(n-1)$-dimensional subspace. Without loss of generality we assume that the $(n-1)$-dimensional subspace is the span of $\{\textbf{e}_3, \ldots, \textbf{e}_{n+1}\}$. For $\tau>0$, let $C_\tau$ denote the cylinder $\{x_1^2+x_2^2 <\tau^2\}$. 

\begin{lem}
\label{lem:small_A^2_classical_cone}
Let $\hat{M}_j$ be a sequence of smooth properly immersed two-sided stable minimal hypersurfaces in $B_{4R}$. Assume that $\hat{M}_j$ converge (as varifolds in $B_{4R}$) to $C$ as $j \to \infty$, where $C$ is a classical cone as above. Then $\int_{B_{2R} \cap \hat{M}_j}|A_{\hat{M}_j}|^2 \to 0$ as $j \to \infty$.
\end{lem}

\begin{proof}
By scale-invariance we may take $4R=2$. By Lemma \ref{lem:A_sheeting_1}, for $\tau>0$ we have that $\hat{M}_j$ converge strongly to $C$ in $B_2 \setminus C_{\frac{\tau}{2}}$.  In particular, given $\tau>0$, we have $\int_{\big(B_1 \setminus C_\tau\big) \cap \hat{M}_j}|A_{\hat{M}_j}|^2  \to 0$ as $j\to \infty$.
Further, 
\begin{equation*}
\begin{aligned}
\int_{\hat{M}_j \cap C_\tau \cap B_1} |A_{\hat{M}_j}|^2 \leq \left(\int_{\hat{M}_j\cap C_\tau \cap B_1} |A_{\hat{M}_j}|^{\frac{2n}{n-2}}\right)^{\frac{n-2}{n}} \left(\int_{\hat{M}_j \cap C_\tau \cap B_1} 1 \right)^{\frac{2}{n}} \\
\leq \left(\int_{\hat{M}_j \cap B_1} |A_{\hat{M}_j}|^{\frac{2n}{n-2}}\right)^{\frac{n-2}{n}} \mathcal{H}^n(\hat{M}_j\cap C_\tau \cap B_1)^{\frac{2}{n}}.
\end{aligned}
\end{equation*}
Lemma \ref{lem:Caccioppoli_A} taken with $k=0$ implies (as argued for \eqref{eq:higher_integrability}, for a dimensional $K(n)$)
\[\int_{\hat{M}_j\cap B_1} |A_{\hat{M}_j}|^{\frac{2n}{n-2}} \leq K(n)  \left( \int_{\hat{M}_j\cap B_{\frac{3}{2}}}|A_{\hat{M}_j}|^2\right)^{\frac{n}{n-2}}.\]
Letting $\eta$ be a fixed function that is equal to $1$ in $B_{\frac{3}{2}}$, is supported in $B_2$, and with $|\nabla \eta|\leq 4$, we find $\int_{\hat{M}_j\cap B_{\frac{3}{2}}}|A_{\hat{M}_j}|^2 \leq \int_{\hat{M}_j\cap B_2} |\nabla \eta|^2 \leq 16 \Lambda(C)$, for all sufficiently large $\ell$ (we used the stability inequality), where we let $\Lambda(C)=\|C\|(B_2) +1$. Therefore, $\left(\int_{B_1} |A_{\hat{M}_j}|^{\frac{2n}{n-2}}\right)^{\frac{n-2}{n}}$ is uniformly bounded for all sufficiently large $\ell$. Next, observe that $C_\tau \cap B_1$ is contained in the union of $\frac{b(n)}{\tau^{n-1}}$ balls of radius $\tau$, where $b(n)$ is a dimensional constant (a rough cover shows that $b(n)=n^{\frac{n-1}{n}}$ works), and in each such ball the $n$-area of $\hat{M}_j$ is at most $k(n)\Lambda(C) \tau^n$, by the monotonicity formula for the area ratio, for a dimensional constant $k(n)$. Hence $\mathcal{H}^n(\hat{M}_j\cap C_\tau \cap B_1)^{\frac{2}{n}} \leq \Big(\Lambda(C) b(n) k(n) \Big)^{\frac{2}{n}} \tau^{\frac{2}{n}}$ and  
\[\int_{B_1 \cap \hat{M}_j}|A_{\hat{M}_j}|^2 \leq \Bigg(\int_{(B_1 \setminus C_\tau) \cap \hat{M}_j}|A_{\hat{M}_j}|^2 \Bigg) + \Lambda(C) b(n) k(n)\tau^{\frac{2}{n}}.\]
As $\tau>0$ can be chosen arbitrarily small, $\int_{B_1 \cap \hat{M}_j}|A_{\hat{M}_j}|^2 \to 0$ as $j \to \infty$.
\end{proof}

This shows that, for all sufficiently large $j$, the quantity $\int_{B_{2R}^{n+1} \cap \hat{M}_j} |A_{\hat{M}_j}|^2$ is at most $\epsilon_0$, so Theorem \ref{thm:eps_reg_A} applies. 
As immediate consequence, we have:

\begin{lem}
\label{lem:A_sheeting_2}
Let $\hat{M}_j$ be a sequence of smooth properly immersed two-sided stable minimal hypersurfaces that converge (as varifolds) in $B^{n+1}_{4R}(0)$ to a classical cone $C$, as $j \to \infty$, with $n\leq 6$. Then $C$ is a sum of hyperplanes with multiplicity, which we describe as a smooth immersion; moreover, $\sup_{B^{n+1}_{\frac{R}{2}}(0) \cap \hat{M}_j} |A_{\hat{M}_j}|  \to 0$ as $j \to \infty$, and $\hat{M}_j$ converge smoothly to $C$ in $B^{n+1}_{\frac{R}{2}}(0)$ as immersions (with $q$ connected components, with $q=\Theta(\|C\|, 0)$).
\end{lem}

\subsection{Closeness to a union of hyperplanes}

With arguments analogous to those in Section \ref{scenario_class_cone}, we will establish:

\begin{lem}
\label{lem:A_sheeting_3}
Let $n\leq 6$ and let $\hat{M}_j$ be a sequence of smooth properly immersed two-sided stable minimal hypersurfaces that converge (as varifolds) in $B^{n+1}_{4R}(0)$, as $j \to \infty$, to a cone $C$ given by a (finite) union of hyperplanes passing through the origin, each taken with an integer multiplicity. Then $\sup_{B^{n+1}_{\frac{R}{2}}(0) \cap \hat{M}_j} |A_{\hat{M}_j}|  \to 0$ as $j \to \infty$, and (describing $C$ as a smooth immersion) $\hat{M}_j$ converge smoothly to $C$ in $B^{n+1}_{\frac{R}{2}}(0)$ as immersions (with $q$ connected components, with $q=\Theta(\|C\|, 0)$).
\end{lem} 

\begin{proof}
Let $C= \sum_{j=1}^N {q_j} |H_j|$, where $q_j \in \N$ and $H_j$ is a hyperplane (through the origin), $q=\sum_{j=1}^N q_j$. For such a cone, we let $S(C) = \cap_{j=1}^N H_j$. Note that if $\text{dim}(S(C)) = n$ then $C$ is a single hyperplane (possibly with multiplicity), while if $\text{dim}(S(C)) = n-1$ then $C$ is a classical cone. Therefore in these two cases, Lemmas \ref{lem:A_sheeting_1} and \ref{lem:A_sheeting_2} give the conclusion. We prove the general statement by (finite) induction on the dimension. Let $\text{dim}(S(C)) =D \in \{0, 1, \ldots n-2\}$ and assume (inductively) that the conclusion has been established for any cone $\tilde{C}$ (given by a union of hyperplanes with multiplicity) for which $\text{dim}(S(\tilde{C})) \in \{D+1, \ldots, n\}$.

Let $p \in \text{spt}\|C\| \setminus \text{S}(C)$. Upon relabeling, we let $H_j$ for $j\in \{1, \ldots , N_1\}$ be the hyperplanes that pass through $p$, and $H_j$ for $j\in \{N_1+1, \ldots N\}$ be those that do not. Since $p \in \text{spt}\|C\| \setminus \text{S}(C)$, necessarily $N_1 <N$.  

Let $\eta_p$ be the translation $\eta_p(x)= x-p$ in $\R^{n+1}$. Let $C_1={\eta_p}_\sharp(\sum_{j=1}^{N_1} {q_j} |H_j|)=\sum_{j=1}^{N_1} {q_j} {\eta_p}_\sharp |H_j|$. Since the line through $0$ and $p$ is contained in $H_j$ for all $j\in \{1, \ldots , N_1\}$, we have $C_1= \sum_{j=1}^{N_1} {q_j}  |H_j|$. This gives immediately $ S(C_1) \supset  S(C)$. This inclusion is strict: in fact, $p\in S(C_1)$, while $p\notin S(C)$. In particular, $\text{dim}(S(C_1)) > \text{dim}(S(C))$. 

\medskip

For any $p \in \text{spt}\|C\| \setminus \text{S}(C)$ we consider a ball $B_r^{n+1}(p)$ sufficiently small to ensure that, for each $j \in \{1, \ldots, N\}$, we have $H_j \cap B_r^{n+1}(p) \not = \emptyset \Rightarrow p \in H_j$. Then the inductive hypothesis can be applied to the sequence $\eta_p(\hat{M}_j \cap B_r^{n+1}(p))$, since it converges in $B_r^{n+1}(0)$ to a cone $C_1$, given by a union of hyperplanes with multiplicity, and with $\text{dim}(S(C_1)) > \text{dim}(S(C))$ (in view of the argument just given). This yields $\sup_{B^{n+1}_{\frac{r}{8}}(p) \cap \hat{M}_j} |A_{\hat{M}_j}|  \to 0$ as $j \to \infty$.

\medskip

Finally, letting $C_{\tau R}= \{x\in \R^{n+1}: \text{dist}_{\R^{n+1}}\big(x, S(C) \cap B_{2R}^{n+1}(0)\big)<\tau R\}$ for $\tau>0$ given, we note that $C_{\tau R}$ can be covered by $\frac{b(D)}{\tau^{D}}$ balls of radius $\tau R$, where $b(D)$ is a constant that depends on $D$ only. Then we argue as in Lemma \ref{lem:small_A^2_classical_cone}. By H\"older's inequality
\[\frac{1}{R^{n-2}}\int_{\hat{M}_j \cap C_{\tau R} \cap B^{n+1}_{2R}} |A_{\hat{M}_j}|^2  \]\[ \leq \left(\frac{1}{R^{n-\frac{2n}{n-2}}}\int_{\hat{M}_j \cap B^{n+1}_{2R}} |A_{\hat{M}_j}|^{\frac{2n}{n-2}}\right)^{\frac{n-2}{n}} \Bigg(\frac{\mathcal{H}^n(\hat{M}_j \cap C_{\tau R} \cap B^{n+1}_{2R})}{R^n}\Bigg)^{\frac{2}{n}}.\]
Using \eqref{eq:higher_integrability} and stability
\[\frac{1}{R^{n-\frac{2n}{n-2}}}\int_{\hat{M}_j\cap B_{2R}^{n+1}} |A_{\hat{M}_j}|^{\frac{2n}{n-2}} \leq K(n)  \left( \frac{1}{R^{n-2}}\int_{\hat{M}_j\cap B^{n+1}_{3R}}|A_{\hat{M}_j}|^2\right)^{\frac{n}{n-2}} \leq \Lambda(C, n),\]
for a constant $\Lambda(C, n)$ determined by $n$ and by the cone $C$. Moreover, by the monotonicity formula, with $\Lambda(C)$ denoting a constant determined by the cone, for all sufficiently large $j$ we have
\[\mathcal{H}^n(\hat{M}_j \cap C_{\tau R} \cap B^{n+1}_{2R}) \leq b(D)\Lambda(C)\tau^{n-D} R^n.\]
Hence
\[\frac{1}{R^{n-2}}\int_{\hat{M}_j \cap C_{\tau R} \cap B^{n+1}_{2R}} |A_{\hat{M}_j}|^2  \leq c(n, D, C) \tau^{\frac{2(n-D)}{n}},\]
which, for $\tau$ sufficiently small, and together with the conclusion obtained above around any $p \in \text{spt}\|C\| \setminus \text{S}(C)$, implies that $\frac{1}{R^{n-2}}\int_{\hat{M}_j \cap B_{2R}^{n+1}(0)} |A_{\hat{M}_j}|^2 \leq \epsilon_0$ for all sufficiently large $j$. Then for such $j$ Theorem \ref{thm:eps_reg_A} applies and the lemma follows.
\end{proof}

\subsection{Tangent cone analysis and conclusion}

\begin{lem}
\label{lem:A_sheeting_both}
Let $M_j$ be a sequence of smooth properly immersed two-sided stable minimal hypersurfaces in an open set $U\subset \R^{n+1}$, with $n\leq 6$, converging (as varifolds) to a (stationary integral) varifold $V$. Let $x \in U$ be such that (at least) one tangent cone to $V$ at $x$ is either a (finite) union of hyperplanes with multiplicity 
(possibly a single hyperplane), or a classical cone. There exists $\rho>0$ such that $V$ is the varifold associated to a smooth immersion in $B_\rho^{n+1}(x)$, and $M_j$ converge smoothly (as immersions) to $V$ in $B_\rho^{n+1}(x)$.
\end{lem}

\begin{proof}
It suffices to find $\rho>0$ such that $\limsup_{j\to \infty} \sup_{M_j \cap B^{n+1}_{2\rho}(x)} |A_{M_j}|<\infty$ and $2\rho < \text{dist}(x, \partial U)$ (after which, standard compactness under $L^\infty$ curvature bounds gives the result). Arguing by contradiction, we assume that this fails for every such $\rho$, hence there exist a subsequence (not relabeled) $M_j$ and associated points $x_j \in M_j$ with $|A_{M_j}(x_j)| \to \infty$ and $x_j \to x$. Upon passing to a further subsequence (determined by the chosen blow up of $V$ at $x$), we find, for each $j$, rescalings $\tilde{M}_j$ of $M_j$ around $x$ that converge to the chosen tangent to $V$ at $x$ (a hyperplane with multiplicity, a classical cone, or a union of hyperplanes with multiplicity) and such that $|A_{\tilde{M}_j}(\tilde{x}_j)| \to \infty$, where $\tilde{x}_j \in \tilde{M}_j$ is the image of $x_j$ via the dilation associated to $j$. For each type of cone (applying respectively Lemma \ref{lem:A_sheeting_1}, Lemma \ref{lem:A_sheeting_2}, or Lemma \ref{lem:A_sheeting_3}) we reach a contradiction.
\end{proof}

The argument that we recall next is a classical procedure (see e.g.~\cite[Section 6]{SchSim}, \cite[Section 8]{WicDG}) that involves tangent cone analysis (in the sense of varifolds), Federer's dimension reducing principle (see e.g.~\cite[Appendix A]{SimonNotes}), and Simons' classification of stable cones (\cite{Simons}, see also \cite[Appendix B]{SimonNotes} and \cite[Section 3]{SSY}), together with Lemma \ref{lem:A_sheeting_both} itself, to show that:

\begin{Prop}
\label{Prop:dim_red}
Let $U$, $M_j$, $V$ be as in the hypotheses of Lemma \ref{lem:A_sheeting_both}. Then any tangent cone to $V$ is a (finite) union of hyperplanes, each with an integer multiplicity. In particular, the conclusion of Lemma \ref{lem:A_sheeting_both} applies at every $x \in \text{spt}\|V\|$.
\end{Prop}

It turns out that it is natural to prove the stronger result that any iterated tangent to $V$ is a finite union of hyperplanes (each with an integer multiplicity). We first recall the relevant notions and facts.

\medskip

If $M_\ell \to V$ and $C_x$ is a tangent cone to $V$ at $x$, there exist $r_j \to 0$ and a subsequence $\ell(j)$ such that $\hat{M}_{\ell(j)} = \lambda_{\big(x, \frac{1}{r_j}\big)}M_{\ell(j)}$ converge (as varifolds) to $C_x$, where $\lambda_{\big(x, \frac{1}{r_j}\big)}$ is the dilation of factor $\frac{1}{r_j}$ centred at $x$, combined with the translation that sends $x$ to $0$, that is $\lambda_{\big(x, \frac{1}{r_j}\big)}(z) = \frac{z-x}{r_j}$.

The spine $S(C)$ of a cone $C$ is the maximal subspace of translation invariance (and coincides with the set of points of maximal density). (It is immediate that when the cone is a union of hyperplanes, the case addressed in Lemma \ref{lem:A_sheeting_3}, the spine is the intersection of these hyperplanes, justifying the notation used there.) We further recall the notion of iterated tangents (to $V$ at $x$), by which we mean the collection of cones $C$ for which there exist cones $C_1, \ldots, C_N$, with $C_N=C$ and $N\in \N$, $N\geq 1$, and points $p_1 \in C_1 \setminus S(C_1), \ldots p_N \in C_N\setminus S(C_N)$ such that $C_m$ is a tangent cone to $C_{m-1}$ at $p_{m-1}$ for $m\geq 2$ and $C_1=C_x$ is a tangent cone to $V$ at $x$. For every $C$ in the space of iterated tangents to $V$ at $x$, we can find $r_j \to 0$ a subsequence $\ell(j)$ and points $z_{\ell(j)} \to x$ (not necessarily lying on $M_{\ell(j)}$) such that $\tilde{M}_{\ell(j)} = \lambda_{\big(z_{\ell(j)}, \frac{1}{r_j}\big)}M_{\ell(j)}$ converge (as varifolds) to $C$, where $\lambda_{\big(z_{\ell(j)}, \frac{1}{r_j}\big)}$ is the dilation of factor $\frac{1}{r_j}$ centred at $z_{\ell(j)}$, combined with the translation that sends $z_{\ell(j)}$ to $0$, that is $\lambda_{\big(z_{\ell(j)}, \frac{1}{r_j}\big)}(z) = \frac{z-z_{\ell(j)}}{r_j}$. (In the case $N=1$ one can take $z_{\ell(j)} = x$, as seen above.)

\begin{proof}[Proof of Proposition \ref{Prop:dim_red}]
We first establish that \emph{for any iterated tangent $C$ the smoothly immersed part of $C$ is stable}. Indeed, in a sufficiently small ambient ball $B^{n+1}_{\rho_y}(y)$, centred at any given $y \in C$ around which $C$ is smoothly immersed, the (dilated) hypersurfaces $\tilde{M}_{\ell(j)}$ converge smoothly (as immersions) to $C$. This is a consequence of Lemma \ref{lem:A_sheeting_both}, as the (unique, in this case) tangent to $C$ at $y$ is of the type prescribed there.
The arbitrariness of $y$ leads to smooth convergence of $\tilde{M}_{\ell(j)}$ to $C$ on the smoothly immersed part of $C$, and thus the stability condition is inherited by the smoothly immersed part of $C$.

\medskip

By Simons' classification, any cone $C$ of dimension $n \in \{2, \ldots, 6\}$ (with vertex at the origin) in $\R^{n+1}$, smoothly immersed away from the origin, which is stable on the smoothly immersed part, must be a union of hyperplanes. From this we will deduce that \emph{the only iterated tangent cones that are smoothly immersed away from the spine are (finite) unions of hyperplanes with multiplicity and classical cones}. Indeed, if the spine dimension is $n$ or $n-1$ then the cone is respectively a single hyperplane with multiplicity or a classical cone. So we assume that $C \setminus S(C)$ is smoothly immersed and stable, with $\text{dim}(S(C)) \leq n-2$. We slice $C$ with affine planes of dimension complementary to the spine $S(C)$, and orthogonal to it. Any such slice is a cone, of dimension at least $2$, since $\text{dim}(S(C)) \leq (n-2)$, and at most $n$, since it is a slice of the $n$-dimensional cone $C$. This slice (in $\R^{n+1-\text{dim}S(C)}$) is a smoothly immersed cone except possibly for an isolated singularity at the vertex; moreover, its regular part inherits stability. By Simons' result the slice is a union of hyperplanes in $\R^{n+1-\text{dim}S(C)}$. Then also $C$ is a union of hyperplanes.

\medskip

A key fact, underlying Federer's dimension reducing principle, is that the spine dimension strictly increases when we take iterated tangents, $\text{dim}(S(C_1)) < \ldots < \text{dim}(S(C_N))$ with the above notation. (This is due to the fact that, choosing a point $y$ away from the spine $S(C)$, the linear subspace spanned by $y$ and $S(C)$ becomes translation invariant for the tangent to $C$ at $y$.) This is used in the following way, to prove that \emph{any iterated tangent must be smoothly immersed away from its spine}.

Assume that a given cone $C$ in the collection of iterated tangents is not smoothly immersed away from its spine $S(C)$, whose dimension we denote by $s$. As $C$ is neither a hyperplane with multiplicity nor a classical cone, we must have $s \in \{0, \ldots, n-2\}$. We consider a tangent cone to $C$ at a non-immersed point in $C \setminus S(C)$, and iterate this step until we find a cone $\hat{C}$ that is smoothly immersed away from its spine. This is achieved after at most $n-s-1$ iterations (thanks to the strict increase in spine dimension, after $n-s-1$ iterations we must have a classical cone or a hyperplane with multiplicity). We let $\tilde{C}$ be the iterated tangent cone for which $\hat{C}$ is a tangent cone at a non-immersed point $y \in \tilde{C} \setminus S(\tilde{C})$. 

As shown above, $\hat{C}$ is either a (finite) union of hyperplanes with integer multiplicity or a classical cone. Lemma \ref{lem:A_sheeting_both} applies to the sequence $\tilde{M}_{\ell(j)}$ that converges to $\tilde{C}$ in a suitably small ball, contradicting that $y$ is a non-immersed point. This concludes the proof of Proposition \ref{Prop:dim_red}.
\end{proof}

\begin{oss}
The conclusion of Proposition \ref{Prop:dim_red} also says that any tangent to $V$ that is a classical cone is in fact a union of hyperplanes (this also follows from Lemma \ref{lem:A_sheeting_both}). We remark that this could be established a priori, employing the framework of integral curvature varifolds (to obtain that if the limit, as varifolds, of a sequence of smooth stable minimal immersions is a classical cone, then it is a union of hyperplanes). If we did that, there would be no need to treat classical cones separately, and Lemma \ref{lem:A_sheeting_2} could be subsumed under Lemma \ref{lem:A_sheeting_3}.
\end{oss}

Theorem \ref{thm:curv_est_6} follows from Proposition \ref{Prop:dim_red} by a contradiction argument.
Using standard compactness arguments (which require the given mass bounds), we assume the existence of a sequence $M_\ell$ of hypersurfaces in $B_4^{n+1}(0)$ (that satisfy the same assumptions as $M$ in the theorem) and, arguing by contradiction, assume that there exists $x_\ell \in M_\ell\cap B_{\frac{1}{2}}$ with $\limsup_{\ell \to \infty} |A_{M_\ell}(x_\ell)| = \infty$. Allard's compactness gives a (subsequential) stationary limit $V$ for $M_\ell$ (without relabelling the subsequence), in the sense of varifolds. By extracting a further subsequence (without relabelling) we also assume $x_\ell \to x \in \text{spt}\|V\|$. Proposition \ref{Prop:dim_red} (and Lemma \ref{lem:A_sheeting_both}) applied at $x$ contradicts $\limsup_{\ell \to \infty} |A_{M_\ell}(x_\ell)| = \infty$.

\medskip

Theorem \ref{thm:Bernstein6} follows by considering $M \cap B_{4R}^{n+1}(p)$ for any chosen $p \in M$, and translating (sending $p$ to $0$). As $R\to \infty$, the estimate in Theorem \ref{thm:curv_est_6} remains valid with the same $\beta$. This forces $A_M(p)=0$. Hence $A\equiv 0$ on $M$ and the result follows.

\part{Towards a compactness theory for branched stable minimal immersions} 
\label{partII}
 
We are interested, in this second part, in a wider class of immersed hypersurfaces $M$: we allow a singular set $\text{Sing}_M$ with locally finite $\mathcal{H}^{n-2}$-measure, or, more generally, vanishing $2$-capacity. Explicitly, for $U \subset \R^{n+1}$ open, and $\Sigma \subset U$ closed (in $U$) with $\text{cap}_2(\Sigma)=0$ (in particular\footnote{We refer to \cite{Evans_book} for details on capacity. In our context, the implication $\mathcal{H}^{n-2}(\Sigma)<\infty \Rightarrow \text{cap}_2(\Sigma)=0$ is implicitly proved in \cite{SchSim} when $\mathcal{H}^{n-2}(\Sigma)=0$ and refined in \cite{WicDG} for the case $\mathcal{H}^{n-2}(\Sigma)<\infty$ using a Federer-Ziemer argument.}, we allow $\Sigma$ to have locally finite $\mathcal{H}^{n-2}$-measure, that is, $\mathcal{H}^{n-2}(\Sigma \cap K)<\infty$ for every $K\subset \subset U$), we let $\iota:S \to U\setminus \Sigma$ be a (smooth) proper immersion, that we assume to be two-sided minimal and stable, with continuous unit normal $\nu$. Denoting by $\overline{M}$ the closure of $M=\iota(S)$ in $U$, we say that $x \in \text{Sing}_M$ if, for every $r>0$, $B_r^{n+1}(x) \cap \overline{M}$ is not the image of a smooth immersion. (In other words, a point in $\Sigma$ is genuinely singular if $M$ cannot be smoothly extended across it, as an immersion.)

As proved in \cite[(1.18) and Section 5]{SchSim}, the stationarity condition (with respect to the area functional) is valid for ambient deformations in $U$, that is, the integral varifold $|\iota_\sharp S|$ is stationary in $U$ (not only in $U\setminus \Sigma$). This follows from a suitable extension of the monotonicity formula, obtained at points in $\Sigma$, giving Euclidean area growth around all points in $\overline{M}\cap U$, combined with a standard capacity argument. (In fact, \cite{SchSim} shows that $\mathcal{H}^{n-1}(\Sigma)=0$ would be sufficient for this.)

\section{Proof of Theorem \ref{thm:eps_reg_tilt}}

\noindent \emph{The tilt function and the relevant PDE}. For a given fixed unit vector, that we assume without loss of generality to be the last coordinate vector $e_{n+1}$, consider the function $g=\left(1-(\nu \cdot e_{n+1})^2\right)^{1/2}$, well-defined on $S$. Clearly, $0\leq g\leq 1$. Letting $\nabla$ denote the metric gradient on $S$, it is immediate that $|\nabla g| \leq \sqrt{1-g^2}|A|$. This follows by direct computation, since 
\[|\nabla (\nu\cdot e_{n+1})| = |(D \nu) (e_{n+1}^T)|\leq |A||e_{n+1} - (\nu\cdot e_{n+1})\nu|=|A|g,\]
where $e_{n+1}^T$ denotes the tangential part of $e_{n+1}$ and $D\nu$ is the shape operator, and
\begin{equation}
 \label{eq:nabla_nu}
 \nabla g = \frac{-(\nu\cdot e_{n+1}) \nabla (\nu\cdot e_{n+1})}{\sqrt{1-(\nu \cdot e_{n+1})^2}}, \text{ that is, }g^2 |\nabla g|^2   = (1-g^2)|\nabla (\nu \cdot e_{n+1})|^2.
 \end{equation} 
We recall the standard Jacobi field equation $\Delta(\nu \cdot e_{n+1}) = -|A|^2 (\nu \cdot e_{n+1})$, or, equivalently,
\[-\Delta (\nu \cdot e_{n+1})^2 = -2 |\nabla (\nu \cdot e_{n+1})|^2 + 2|A|^2 (\nu \cdot e_{n+1})^2,\]
where $\Delta$ is the Laplace-Beltrami operator on $S$. This implies a PDE for $g$ on $S$ (using the relation \eqref{eq:nabla_nu}):  
\[(1-g^2)(2g\Delta g + 2|\nabla g|^2) = (1-g^2)\big(\Delta g^2\big)= -2  g^2 |\nabla g|^2 + 2|A|^2 (1-g^2)^2,\]
and therefore (in view of $|\nabla g|^2 \leq (1-g^2)|A|^2$ the following is well-defined) 
\begin{equation}
 \label{eq:PDE_for_g}
g \Delta g=  -\frac{|\nabla g|^2}{1-g^2} + |A|^2 (1-g^2).
\end{equation}

We recall that the following improved inequality (see \cite[(2.7)]{SchSim}) is implied by the minimality condition:
\begin{equation}
 \label{eq:improved_inequality_nablag}
 \frac{|\nabla(\nu \cdot e_{n+1})|^2}{(1-(\nu\cdot e_{n+1})^2)} \leq \left(1-\frac{1}{n}\right) |A|^2  \Leftrightarrow \frac{|\nabla g|^2}{1-g^2} \leq \left(1-\frac{1}{n}\right) |A|^2.
\end{equation}

\begin{oss}
The quantity $E_M(R)^2=\frac{1}{R^n}\int_{C_R} g^2$ (appearing in Theorem \ref{thm:eps_reg_tilt}) is the square of the scale-invariant tilt-excess of $|M|$ in $C_{R}=B^n_R(0)\times (-R, R)$, with respect to the hyperplane $\R^n\times \{0\}$ (orthogonal to $e_{n+1}$). As in Part \ref{partI}, with slight notational abuse we will write domains $D$, or $M\cap D$, to mean $\iota^{-1}(D)$, where $\iota:S \to C_{2R}$ is the immersion with image $M$. 
\end{oss}
\begin{oss}[\emph{height and tilt excess}]
\label{oss:height_excess}
We recall that the (scale-invariant) $L^2$ height-excess $\hat{E}_M(r)$, defined by $\hat{E}_M(r)^2=\frac{1}{r^{n+2}}\int_{M\cap C_{r}} |x_{n+1}|^2$, bounds (linearly) $E_M(\frac{r}{2})^2$. Indeed, stationarity implies, using the first variation formula with a vector field $x_{n+1} \varphi^2 e_{n+1}$, for $\varphi\in C^1_c(C_{2R})$ taken to be identically $1$ in $C_R$ and with $|\nabla \varphi|\leq \frac{1}{R}$, the validity of the inequality (see e.g.~\cite[Section 22]{SimonNotes})
\[\frac{1}{R^n}\int_{M\cap C_R} |\nabla x_{n+1}|^2  \leq \frac{2^{n+4}}{(2R)^{n+2}}\int_{M\cap C_{2R}} |x_{n+1}|^2,\]
and $|\nabla x_{n+1}| = |\text{proj}_{TM}(e_{n+1})| = \sqrt{1-(\nu\cdot e_{n+1})^2} =g$.
\end{oss} 

The proof of Theorem \ref{thm:eps_reg_tilt} will be carried out by means of an iteration \`{a} la De Giorgi, for which the fundamental lemma is an intrinsic weak Caccioppoli inequality, for level set truncations of $g$ (Lemmas \ref{lem:De_Giorgi_class_subsol_modified} and \ref{lem:Caccioppoli_g_sing} below).  

\begin{lem}
\label{lem:De_Giorgi_class_subsol_modified} 
Let $M$ be as above. Then for any $k\in [0,\frac{1}{2n}]$ and $\phi\in C^{0,1}_c(S)$ we have
\[\frac{1}{2n}\int_{\{g>k\}} |\nabla g|^2 \left(1-\frac{k}{g} \right) \phi^2 \leq  \int {(g-k)^+}^2 |\nabla \phi|^2,\]
where $(g-k)^+$ denotes the function $(g-k)^+=\left\{\begin{array}{ccc}         g - k & \text{ when } g> k\\
0 & \text{ when } g\leq k              
\end{array} \right.$.
\end{lem}

\begin{proof}
We use the stability condition, whose analytic form is the validity of 
\[\int |A|^2 \eta^2 \leq \int|\nabla \eta|^2\]
for all $\eta\in C^1_c(S)$. A standard approximation argument implies that $\eta\in C^{0,1}_c(S)$ is also allowed and we choose $\eta = (g-k)^+ \phi$, where $\phi\in C^{0,1}_c(S)$ (as in the statement). We compute (on the right-hand-side of the stability inequality)
\begin{equation*}
 \begin{aligned}
 \int |\nabla((g-k)^+\phi)|^2 = \\
 \int|\nabla(g-k)^+|^2 \phi^2 + \underbrace{2 \int (g-k)^+\, \phi\, \nabla(g-k)^+ \, \nabla\phi}_{\frac{1}{2}\int \nabla({(g-k)^+}^2) \, \nabla(\phi^2)} + \int {(g-k)^+}^2 |\nabla \phi|^2.
 \end{aligned}
\end{equation*}
We note that the function ${(g-k)^+}^2$ is in $C^1\cap W^{2,\infty}(S)$. Indeed, $\nabla({(g-k)^+}^2)=2{(g-k)^+}\nabla g$ and this function is locally Lipschitz. In particular, we have that $\Delta({(g-k)^+}^2)$ is the $L^\infty$ function that vanishes in the complement of $\{g\geq k\}$ and is equal to  $2 {(g-k)^+} \Delta{(g-k)^+} + 2|\nabla({(g-k)^+}|^2$ on $\{g<k\}$. Hence we can integrate by parts and the braced term becomes
\[-\frac{1}{2}\int \Delta({(g-k)^+}^2) \, \phi^2 = - \int |\nabla{(g-k)^+}|^2 \, \phi^2 - \int_{\{g>k\}} {(g-k)^+} \Delta{(g-k)^+} \, \phi^2.\]
The right-hand-side of the stability inequality is therefore
\begin{equation*}
 \begin{aligned}
- \int_{\{g>k\}} (g-k) \Delta g \, \phi^2 + \int {(g-k)^+}^2 |\nabla \phi|^2 \stackrel{\text{ by }\eqref{eq:PDE_for_g}}{=}\\
\int_{\{g>k\}} \Bigg(1-\frac{k}{g}\Bigg) \frac{|\nabla g|^2}{1-g^2} \, \phi^2 - \int_{\{g>k\}} \Bigg(1-\frac{k}{g}\Bigg) |A|^2 (1-g^2) \, \phi^2+ \int {(g-k)^+}^2 |\nabla \phi|^2.
 \end{aligned}
\end{equation*}
(When $k=0$ we do not need to multiply the PDE \eqref{eq:PDE_for_g} by $\frac{g-k}{g}=1-\frac{k}{g}$.) We now use the improved inequality \eqref{eq:improved_inequality_nablag} (for the first integrand in the last expression) and find, from the stability inequality,
\begin{equation*}
 \begin{aligned}  
\int_{\{g>k\}} |A|^2 (g-k)^2 \phi^2 \leq \int_{\{g>k\}} \Bigg(1-\frac{1}{n}\Bigg) \Bigg(1-\frac{k}{g}\Bigg) |A|^2 \, \phi^2\\
 - \int_{\{g>k\}} \Bigg(1-\frac{k}{g}\Bigg) |A|^2 (1-g^2) \, \phi^2+ \int {(g-k)^+}^2 |\nabla \phi|^2.
 \end{aligned}
\end{equation*} 
Moving all terms containing $|A|^2$ to the left-hand-side we compute
\begin{equation*}
 \begin{aligned}
(g-k)^2 - \Bigg(1-\frac{1}{n}\Bigg) \Bigg(1-\frac{k}{g}\Bigg) +\Bigg(1-\frac{k}{g}\Bigg) (1-g^2)=\\
(g-k)\Bigg(g-k -\frac{1}{g}\Big(1-\frac{1}{n} -1 + g^2\Big)\Bigg) = \frac{g-k}{g}\bigg(g^2-kg + \frac{1}{n} - g^2\bigg),
 \end{aligned}
\end{equation*} 
which gives
\[\int_{\{g>k\}} |A|^2 \,\frac{(g-k)}{g}\left(\frac{1}{n} - kg\right)\, \phi^2 \leq  \int {(g-k)^+}^2 |\nabla \phi|^2.\]
As $g\in[0,1]$, if we restrict $k\in[0,\frac{1}{2n}]$ as in the hypotheses we get $\frac{1}{n} - kg \geq \frac{1}{2n}$, hence
\begin{equation}
 \label{eq:Schoen_ineq_k}
\frac{1}{2n}\int_{\{g>k\}} |A|^2 \left(1-\frac{k}{g} \right) \phi^2 \leq  \int {(g-k)^+}^2 |\nabla \phi|^2.
\end{equation}
Using $|\nabla g|\leq |A|$ 
we reach
\[\frac{1}{2n}\int_{\{g>k\}} |\nabla g|^2 \left(1-\frac{k}{g} \right) \phi^2 \leq  \int {(g-k)^+}^2 |\nabla \phi|^2.\]
\end{proof}

\begin{lem}
\label{lem:Caccioppoli_g_sing} 
Let $M$ be as above. Then for any $k\in [0,\frac{1}{2n}]$ and $\varphi\in C^{0,1}_c(U)$ we have
\[\frac{1}{2n}\int_{M\cap \{g>k\}} |\nabla g|^2 \left(1-\frac{k}{g} \right) \varphi^2 \leq  \int_M {(g-k)^+}^2 |\nabla \varphi|^2,\]
where $(g-k)^+$ denotes the function $(g-k)^+=\left\{\begin{array}{ccc} 
g - k & \text{ when } g> k\\
0 & \text{ when } g\leq k        
\end{array} \right.$.
\end{lem}

\begin{proof}
The statement is just Lemma \ref{lem:De_Giorgi_class_subsol_modified} when $\varphi \in C^{0,1}_c(U \setminus \Sigma)$. (We are implicitly choosing $\phi=\varphi \circ \iota$; the immersion is proper so $\varphi \circ \iota \in C^1_c(S)$.) Taking this as starting point, the extension of the inequality to $\varphi\in C^{0,1}_c(U)$ relies on the Euclidean area growth of $n$-area (valid at all points in $\overline{M}$, as recalled above) and on the assumption that $\text{cap}_2(\Sigma)=0$. The (now standard) $2$-capacity argument is carried out in \cite{SchSim} for the case $\mathcal{H}^{n-2}(\Sigma)=0$ and in \cite{WicDG} for the case $\mathcal{H}^{n-2}(\Sigma)<\infty$.
\end{proof}

\begin{oss}
In the case $k=0$, \eqref{eq:Schoen_ineq_k} is the well-known Schoen inequality, \cite{Sch}, \cite[Lemma 1]{SchSim}. The instance $k=0$ of the lemma gives the intrinsic Caccioppoli inequality $\frac{1}{2n}\int |\nabla g|^2  \varphi^2 \leq  \int g^2 |\nabla \varphi|^2$.
\end{oss}

We will employ Lemma \ref{lem:Caccioppoli_g_sing}, with suitable choices of $\varphi$.
We will obtain a superlinear rate of decay for the $L^2$-norm of $(g-k)^+$ in $C_{r}$ as $k\in \R$ grows from $0$ to $\frac{1}{2n}$, and $r$ decreases from the initial scale $R$ to $\frac{R}{2}$. Via the elementary Lemma \ref{lem:analysis1}, such a decay forces $(g-\frac{1}{2n})^+$ to vanish in $C_{\frac{R}{2}}$, as long as the $L^2$-norm of $g$ is sufficiently small in $C_R$. This will establish Theorem \ref{thm:eps_reg_tilt}.

\begin{proof}[proof of Theorem \ref{thm:eps_reg_tilt} for $n\geq 3$]
We consider the dyadic sequences (respectively increasing and decreasing) $k_\ell = \frac{d}{2n}\left(1-\frac{1}{2^{\ell-1}}\right)$ for $d\in(0,1]$ (for the moment arbitrary), and $R_\ell = \frac{R}{2} + \frac{R}{2^\ell}$ for $\ell \in \{1, 2, \ldots \}$.

We take the inequality of Lemma \ref{lem:Caccioppoli_g_sing} with $k_\ell$ in place of $k$, and use the inclusion $\{g>k_\ell\}\supset \{g>k_{\ell+1}\}$:
\[\frac{1}{2n}\int_{\{g>k_{\ell+1}\}} |\nabla g|^2 \left(1-\frac{k_\ell}{g} \right) \varphi^2 \leq  \int {(g-k_\ell)^+}^2 |\nabla \varphi|^2;\]
on the relevant domain on the left-hand-side, $\{g>k_{\ell+1}\}$, we have  $1-\frac{k_\ell}{g}\geq 1-\frac{k_\ell}{k_{\ell+1}} \geq \frac{1}{2^\ell}$, therefore
\[\int_{\{g>k_{\ell+1}\}} |\nabla g|^2  \varphi^2 \leq (2n)  2^\ell  \int {(g-k_\ell)^+}^2 |\nabla \varphi|^2.\]
Since $|\nabla \big((g-k_{\ell+1})^+ \varphi\big)|^2 \leq 2\chi_{\{g>k_{\ell+1}\}}|\nabla g|^2 \varphi^2 + 2{(g-k_{\ell+1})^+}^2 |\nabla \varphi|^2$, and by definition $(g-k_{\ell+1})^+ \leq (g-k_{\ell})^+$, we find
\begin{equation}
\label{eq:estimate_before_MS} 
\int_{\{g>k_{\ell+1}\}} \Big|\nabla \big((g-k_{\ell+1})^+ \varphi \big)\Big|^2   \leq 2((2n)  2^\ell +1)  \int {(g-k_\ell)^+}^2 |\nabla \varphi|^2.
\end{equation}
Next (from now we will use $n\geq 3$), we will use (for the left-hand-side of \eqref{eq:estimate_before_MS}) the following Michael-Simon inequality (\cite{Bre,MS}) on the minimally immersed hypersurface $M$, for the function $\varphi (g-k_{\ell+1})^+$:
\begin{equation}
 \label{eq:Michael_Simon_ineq}
\left(\int \left| \varphi  (g-k_{\ell+1})^+ \right|^{\frac{2n}{n-2}}\right)^{\frac{n-2}{n}}\leq C_{MS} \int \left|\nabla \left(\varphi  (g-k_{\ell+1})^+\right) \right|^{2},
\end{equation}
with $C_{MS}$ the dimensional constant given after \eqref{eq:MichaelSimon_A}. 

Simultaneously, we choose $\varphi$, as follows. For $r >\rho$ chosen in $(0,R]$ we will consider $\varphi$ of the type $\varphi(x, x_{n+1})=\tilde{\varphi}(x)\tilde{\psi}(x_{n+1})$, with $\tilde{\varphi}:\R^n \to \R$ identically equal to $1$ on $B_{\rho}^{n}(0)$, vanishing in the complement of $B_{r}^{n}(0)$ and with $|D \tilde{\varphi}|\leq \frac{\sqrt{2}}{r-\rho}$; with $\tilde{\psi} \in C^\infty_c(\R)$ identically equal to $1$ on $[-\rho, \rho]$, vanishing in the complement of $(-r, r)$, with $|\tilde{\psi}^\prime|\leq \frac{\sqrt{2}}{r-\rho}$. Then, for each $\ell$, we choose $\tilde{\varphi}_\ell$ and $\tilde{\psi}_\ell$ with $\rho=R_{\ell+1}$, $r=R_\ell$, so that $r-\rho=R_{\ell}-R_{\ell+1}=\frac{R}{2^{\ell+1}}$, and $\varphi_\ell=\tilde{\varphi}_\ell \tilde{\psi}_\ell$. Note that $|\nabla \varphi_\ell|\leq \frac{2}{R_{\ell}-R_{\ell+1}}$, $\varphi_\ell \equiv 1$ on $C_{R_{\ell+1}}$ and $\varphi_\ell \equiv 0$ in the complement of $C_{R_{\ell}}$. 
Combining \eqref{eq:estimate_before_MS} and \eqref{eq:Michael_Simon_ineq}, with the chosen $\varphi_\ell$ in place of $\varphi$, we find
\begin{equation}
 \label{eq:estimate_after_MS}
 \begin{aligned}
\left(\int_{M\cap C_{R_{\ell+1}}} \left| (g-k_{\ell+1})^+ \right|^{\frac{2n}{n-2}}\right)^{\frac{n-2}{n}}\leq \left(\int \left| \varphi_\ell  (g-k_{\ell+1})^+ \right|^{\frac{2n}{n-2}}\right)^{\frac{n-2}{n}}\leq \\
\leq C_{MS}(n) \frac{2((2n)  2^\ell +1) 4^{\ell+2}}{R^2} \int_{M\cap C_{R_{\ell}}} {(g-k_\ell)^+}^2.
\end{aligned}
\end{equation}
H\"older's inequality further gives  
\begin{equation}
 \label{eq:Holder_ineq_g}
 \begin{aligned}
\int_{M\cap C_{R_{\ell+1}}} {(g-k_{\ell+1})^+}^{2}\leq \\
\left(\int_{M\cap C_{R_{\ell+1}}} {(g-k_{\ell+1})^+}^{\frac{2n}{n-2}}\right)^{\frac{n-2}{n}} \mathcal{H}^n\Big(M \cap \{g>k_{\ell+1}\} \cap C_{R_{\ell+1}} \Big)^{\frac{2}{n}}.
 \end{aligned}
\end{equation}
Noting that on the set $\{g>k_{\ell+1}\}$ we have $(g-k_\ell)^+ > \frac{d}{n\,2^{\ell+1}}$, and using the inclusion $C_{R_{\ell+1}} \subset C_{R_{\ell}}$, the last factor in \eqref{eq:Holder_ineq_g} is bounded above (thanks to the standard Markov's inequality) by
\begin{equation}
\label{eq:Markov_g}
 \begin{aligned}
\mathcal{H}^n\Bigg(M\cap \left\{{(g-k_{\ell})^+}^2>\frac{d^2}{(n\,2^{\ell+1})^2}\right\} \cap C_{R_{\ell}} \Bigg)^{2/n}\leq  \\
\left(\frac{n^2\,4^{\ell+1}}{d^2} \int_{M\cap C_{R_{\ell}}}{(g-k_\ell)^+}^2 \right)^{2/n}.
 \end{aligned}
\end{equation}
From \eqref{eq:Holder_ineq_g}, using \eqref{eq:Markov_g} for the second factor on the right-hand-side, and using \eqref{eq:estimate_after_MS} for the first factor on the right-hand-side, we have
\begin{equation*}
 \begin{aligned}
\int_{M\cap C_{R_{\ell+1}}} {(g-k_{\ell+1})^+}^{2}\leq \\
C_{MS}(n) \frac{8((2n)  2^\ell +1)4^{\ell+1}}{d^{\frac{4}{n}} R^2} n^{\frac{4}{n}}(4^{\frac{2}{n}})^{\ell+1} \Big(\int_{M\cap C_{R_{\ell}}} {(g-k_\ell)^+}^2\Big)^{1+\frac{2}{n}}.
 \end{aligned}
\end{equation*}
Writing $G_\ell=\int_{C_{R_{\ell}}} {(g-k_{\ell})^{+}}^2$, this implies
\begin{equation}
 \label{eq:decay_estimate}
G_{\ell+1} \leq \underbrace{C_{MS}(n)\frac{32 n}{d^{\frac{4}{n}} R^2}n^{4/n} (4\cdot 4^{2/n})}_{c(n, R, d)}\, (8\cdot 4^{2/n})^{\ell} \,G_{\ell}^{1+\frac{2}{n}}.
\end{equation}
The superlinear decay estimate \eqref{eq:decay_estimate}, $G_{\ell+1}\leq c(n, R, d) C^\ell G_{\ell}^{1+\frac{2}{n}}$ with $C=2\cdot 4^{1+\frac{2}{n}}$, forces $G_\ell\to 0$ as $\ell\to \infty$, as long as $G_1$ is sufficiently small, in a quantified fashion determined by $c(n, R, d)$ and $C$. We make it now explicit, using Lemma \ref{lem:analysis1}.

With the initial choice $d=1$, the smallness condition on $G_1$ is written, for the scale-invariant tilt-excess ($R_1=R$ so $E_M(R)^2=\frac{1}{R^n}G_1$), as
\begin{equation}
 \label{eq:dim_constant_1}
E_M(R)^2\leq \frac{1}{\left(R^2 c(n, R, 1) C^{\frac{n+2}{2}}\right)^{n/2}} = \left(\frac{1}{C_{MS}(n) 32 \,n^{1+\frac{4}{n}}  4^{1+\frac{2}{n}} 2^{\frac{n+2}{2}}  4^{\frac{(n+2)^2}{2n}}}\right)^{n/2},
\end{equation}
where the last term makes explicit the dimensional constant $k(n)$ in Theorem \ref{thm:eps_reg_tilt}.

The convergence $G_\ell\to 0$ implies $\int_{C_{\frac{R}{2}}} {(g-\frac{1}{2n})^{+}}^2=0$, that is, $g\leq \frac{1}{2n}$ a.e.~on $M\cap C_{\frac{R}{2}}$. By smoothness of $\iota$, then $g\leq \frac{1}{2n}$ on $M\cap C_{\frac{R}{2}}$.

More generally, with $d\in (0, 1]$, we find that, if
\begin{equation*}
E_M(R)^2\leq \frac{1}{\left(R^2 c(n, R, d) C^{\frac{n+2}{2}}\right)^{n/2}} = \frac{d^{2}}{\Big(C_{MS}(n) 32 \,n^{1+\frac{4}{n}}  4^{1+\frac{2}{n}} 2^{\frac{n+2}{2}}  4^{\frac{(n+2)^2}{2n}}\Big)^{n/2}},
\end{equation*}
then $g\leq \frac{d}{2n}$ on $M\cap C_{\frac{R}{2}}$. In other words, we have proved that, in the regime $E_M(R)^2\leq k(n)$, we have (for an explicit dimensional constant $c(n)$)
\[\sup_{M\cap C_{\frac{R}{2}}} g \leq c(n) E_M(R).\]
\end{proof}

\begin{oss}
\label{oss:CI_g}
For $k=0$, the inequality in Lemma \ref{lem:Caccioppoli_g_sing} is an \emph{intrinsic} Caccioppoli inequality (we have the intrinsic gradient on $M$, rather than the gradient $D$ in $\R^n$ as in the case of De Giorgi \cite{DG}, see also Remark \ref{oss:CI_A}). For $k \in (0, \frac{1}{2n}]$, on the other hand, we only have a weak intrinsic Caccioppoli inequality, due to the multiplicative factor $\left(1-\frac{k}{g} \right)^+$. As seen also for Lemma \ref{lem:Caccioppoli_A}, this weaker inequality is sufficient to implement the iterative scheme. While in \cite{DG} it is the linearity of the PDE that permits to obtain the classical Caccioppoli inequality for $(u-k)^+$, in our case the PDE for $g$ escapes the De Giorgi--Nash--Moser framework: in fact, the PDE is a consequence of the minimality of $M$ alone, which would permit e.g.~catenoidal necks, with $g$ reaching the value $1$ under any smallness assumption on the $L^2$ height- or tilt-excess. The stability condition provides sufficient control on the non-linearity of the PDE (\ref{eq:PDE_for_g}) to obtain the weak intrinsic Caccioppoli inequality. We note explicitly that Lemma \ref{lem:Caccioppoli_g_sing} is only valid for truncations at sufficiently small level sets (hence the smallness requirement in Theorem \ref{thm:eps_reg_tilt}). 
\end{oss}

\begin{proof}[proof of Theorem \ref{thm:eps_reg_tilt} for $n=2$]
The case $n=2$ requires a modification, as the exponent $\frac{2n}{n-2}$ is not well-defined in that case. The choices of $k_\ell$, $R_\ell$, $\varphi_\ell$ remain the same. We start from \eqref{eq:estimate_before_MS} (only after which we used $n\geq 3$), choosing $\varphi$ in \eqref{eq:estimate_before_MS} to be $\varphi_\ell$ (recall that $|\nabla \varphi_\ell|\leq \frac{2^{\ell+1}}{R}$ and that $\text{spt} \varphi_\ell \subset C_{R_\ell}$). In what follows, $\sigma, \sigma', \sigma''$ denote explicitly determinable constants. We have 
\[\int \Big|\nabla \big((g-k_{\ell+1})^+ \varphi_\ell \big)\Big|^2   \leq \frac{4^\ell \sigma}{R^2} \int_{C_{R_\ell} \cap M} {(g-k_\ell)^+}^2 .\]
We use H\"older's inequality
\[\Big(\int  \Big|\nabla \big((g-k_{\ell+1})^+ \varphi_\ell \big)\Big|\Big)^{2}  \leq  \Big(\int  \Big|\nabla \big((g-k_{\ell+1})^+ \varphi_\ell \big)\Big|^2\Big)  \mathcal{H}^2\big(M\cap \{g>k_{\ell+1}\} \cap C_{R_\ell}\big), \]
the Michael--Simon inequality 
\[\int  \big( (g-k_{\ell+1})^+ \varphi \big)^2 \leq C_{MS}^2 \Big(\int \Big|\nabla \big((g-k_{\ell+1})^+ \varphi \big)\Big|\Big)^{2}\]
and the following consequence of Markov's inequality (as justified earlier)
\[\mathcal{H}^2\big(M \cap \{g>k_{\ell+1}\}\cap  C_{R_{\ell}}  \big) \leq  \frac{\sigma' 4^{\ell}}{d^2} \int_{M\cap C_{R_{\ell}}}{(g-k_\ell)^+}^2 .\]
Writing $G_\ell = \int_{M\cap C_{R_{\ell}}}{(g-k_\ell)^+}^2$, combining these inequalities we find
\[G_{\ell+1} \leq \frac{16^\ell \sigma''}{R^2 d^2} G_\ell^2.\]
At this stage, Lemma \ref{lem:analysis1} gives that, if $G_1 \leq  \frac{R^2 d^2}{16 \sigma''}$ then $G_\ell \to 0$ as $\ell \to \infty$. In other words, given $d\in (0,1]$, if $E_M^2(R)=\frac{1}{R^2}\int_{M\cap C_R} g^2 \leq \frac{d^2}{16 \sigma''}$, then $\sup_{M\cap C_{\frac{R}{2}}} g\leq \frac{d}{2n}$. The conclusion of Theorem \ref{thm:eps_reg_tilt} is thus proved for $n=2$: in the smallness regime $E_M^2(R) \leq \frac{1}{16 \sigma''}$ we have the control $\sup_{M\cap C_{\frac{R}{2}}} g\leq c(n)E_M(R)$.
\end{proof}

\section{Proof of Theorems \ref{thm:immersion_sing_sheeting}, \ref{thm:immersion_smooth_sheeting},  \ref{thm:SchSim}}

\begin{proof}[Proof of Theorem \ref{thm:immersion_smooth_sheeting}]
The pointwise bound $g\leq \frac{1}{2n}$ obtained in Theorem \ref{thm:eps_reg_tilt} implies the decomposition result by elementary arguments. 
Being an immersion, $\iota$ is locally a diffeomorphism with its image, that is, for every $X\in S$ there exists a neighbourhood $D_X$ such that $\left.\iota\right|_{D_X}$ is an embedded disk. The bound on $g$ implies that that there exists a choice of continuous unit normal $\nu$ such that $\nu\cdot e_{n+1} \geq \frac{\sqrt{(2n)^2 -1}}{2n}$. Denote by $\pi$ the projection $\R^n\times \R \to \R^n$. Then (for $D_X$ sufficiently small) the disk $\iota(D_X)$ is a smooth graph over its projection. We thus have that $\left.\iota\right|_{\iota^{-1}(C_{R/2})}$ is a local diffeomorphism with $B^n_{R/2}(0)$. Fix a connected component of $\iota^{-1}(C_{R/2})$, which we denote $S_0$. Then $\left.\iota\right|_{S_0}: S_0 \to B^n_{R/2}(0)$ is a local diffeomorphism. 

The condition on $\nu\cdot e_{n+1}$ guarantees that $\left.\iota\right|_{S_0}$ is transverse to any line of the form $\{q\}\times \R$ (oriented by $e_{n+1}$) and the intersection is always positive. Moreover, the intersection index of $M$ with such lines is constant (since $\sup_{M \cap C_{\frac{R}{2}}} |x_{n+1}| < \frac{R}{2}$, $M \cap C_{\frac{R}{2}}$ has no boundary in $B^n_{\frac{R}{2}}(0) \times \R$). Therefore $\iota^{-1}(\{q\}\times \R)$ is a subset of $S_0$ with fixed cardinality $N\in \N$, regardless of $q$. (The immersion is proper, therefore there can only be finitely many points of intersection.) 

The above observations imply that $\left.\iota\right|_{S_0}$ is a $N$-cover of $B^n_{R/2}(0)$. On the other hand, the ball $B^n_{R/2}(0)$ is its own universal cover (and $S_0$ is connected), so $N=1$. We have proved that each connected component of $\iota^{-1}(C_{R/2})$ is mapped (by $\iota$) to a (smooth) graph over $B^n_{R/2}(0)$, which provides the smooth functions $v_j$ in the conclusion of Theorem \ref{thm:immersion_smooth_sheeting} (where $j$ ranges over the set of connected components, which are finitely many because $\iota$ is proper).

\medskip

At this stage, one can follow the arguments of De Giorgi \cite{DG}, or directly invoke the De Giorgi--Nash--Moser theory, to conclude that $g$ is H\"older continuous on every $\text{graph}(v_j)$, and that each $v_j$ is in $C^{1,\alpha}(B^n_{R/2}(0))$, with the estimate $\|\nabla v_j\|_{C^{0,\alpha}(B^n_{R/2}(0))} \leq C(n) E_0$. Higher regularity (and the analogous estimate for the $C^{k,\alpha}$-norms) follow from Schauder theory (using the Schoen inequality to control the $L^2$ norm of $A$ by the tilt excess). 
\end{proof}

\begin{proof}[Proof of Theorem \ref{thm:immersion_sing_sheeting}]
The arguments given for the graph decomposition for Theorem \ref{thm:immersion_smooth_sheeting} lead to the conclusion that $\iota$ restricted to any connected component $S_0$ of $S$ is a local diffeomorphism and an $N$-cover of $B^n_{R/2}(0) \setminus \pi(\Sigma)$. This relies on the observation that $B^n_{R/2}(0) \setminus \pi(\Sigma)$ is open and (path) connected (a consequence of the fact that $\Sigma$ is closed with $\mathcal{H}^{n-1}(\Sigma)=0$, which follows from $\text{cap}_2(\Sigma)=0$). This guarantees the possibility to choose a normal that has positive intersections with lines $\{q\}\times \R$ and the constancy of the intersection index of $M$ with such lines for $q\in B^n_{R/2}(0) \setminus \pi(\Sigma)$. 
 
At this stage, we have a description of $M \cap C_{R/2}$ as graph of a smooth $q$-valued function on $B^n_{R/2}(0) \setminus \pi(\Sigma)$. For any $q \in B^n_{R/2}(0)\setminus \pi(\Sigma)$, by ordering the values $\Pi \Big(M\cap (\{q\}\times \R)\Big) \subset \R$ increasingly, where $\Pi$ is the projection onto the second factor of $B^n_{R/2}(0) \times \R$, we obtain $q$ Lipschitz functions $u_j:B^n_{R/2}(0) \setminus \pi(\Sigma) \to \R$, with Lipschitz constant $\frac{1}{2n}$, $u_j \leq u_{j+1}$ for all $j \in \{1, \ldots, Q-1\}$, which we can extend (preserving the Lipschitz constant) to $u_j:B^n_{R/2}(0) \to \R$.
\end{proof}

Theorem \ref{thm:immersion_sing_sheeting} gives a sheeting theorem for immersions, that are allowed to possess a singular set of locally finite $\mathcal{H}^{n-2}$-measure (or vanishing $2$-capacity), and that are assumed to be ``close'' to a hyperplane. For such immersions, genuine branch points may arise, hence the singular set cannot be ruled out in the conclusions.

\begin{proof}[Proof of Theorem \ref{thm:SchSim}]
Specialising Theorem \ref{thm:immersion_sing_sheeting} to embeddings, that is, if $\iota(M)$ is properly embedded in $C_R \setminus \Sigma$, then the Lipschitz functions $u_j:B^n_{R/2}(0) \setminus \pi(\Sigma) \to \R$ must be such that, for every $j \in \{1, \ldots, Q-1\}$, $u_j < u_{j+1}$. Thanks to the strict inequality, each $u_j$ is a Lipschitz solution of the weak minimal surface PDE on $B^n_{R/2}(0) \setminus \pi(\Sigma)$, hence a smooth strong solution. Simon's well-known singularity removal \cite{Sim77}, which only requires $\pi(\Sigma)$ closed in $B^n_{R/2}(0)$ and $\mathcal{H}^{n-1}(\pi(\Sigma))=0$ (a consequence of $\text{cap}_2(\Sigma)=0$), yields a smooth extension $u_j:B^n_{R/2}(0) \to \R$ for each $j$, so that $\text{sing}_M \cap C_{\frac{R}{2}}=\emptyset$. 
\end{proof}

\begin{oss}
As shown in \cite{SchSim}, Theorem \ref{thm:SchSim} leads rather quickly to the renowned Schoen--Simon regularity and compactness theory for stable minimal embedded hypersurfaces, see \cite[Theorems 2 and 3]{SchSim}.

The extra step required for this is a fairly simple slicing argument, see \cite[pp. 785--787]{SchSim}, which proves that ``closeness'' to a classical cone cannot arise for embeddings; after that, standard tangent cone analysis and dimension reduction complete the proof. For contrast, in the immersed case, closeness to classical cones can arise (and one would naturally aim for a sheeting result, over the several hyperplanes constituting the classical cone, which for $n\leq 6$ and in the absence of singular set follows from Lemma \ref{lem:A_sheeting_2} of Part \ref{partI}).
\end{oss}

With the multi-valued description of $M$ in Theorem \ref{thm:immersion_sing_sheeting}, natural questions are a more precise characterisation of the $q$-valued function obtained (plausibly, one can establish $C^{1,\alpha}$ regularity in the sense of $q$-valued functions), and a finer structure result for the singular set. While we do not pursue this here, we observe:

\begin{cor}[uniqueness of tangent hyperplanes]
\label{cor:uniq_tang}
Let $M$ be as in the beginning of Part \ref{partII}, and let $x \in \overline{M}$ be such that there exists a tangent cone (in the sense of varifolds) to $M$ at $x$ that is a hyperplane with multiplicity. Then that is the unique tangent cone at $x$. 
\end{cor}

\begin{proof}
We take a blow up that gives rise to a hyperplane with multiplicity, which we assume to be $\{x_{n+1}=0\}$ by rotating coordinates. For the blow up sequence $M_\ell$ (obtained by dilations of $M$) we have $E_{M_\ell}(1) \to 0$ (this follows from the monotonicity formula, using also Remark \ref{oss:height_excess}). Denoting by $g_\ell$ the tilt function on $M_\ell$, using the estimate $\sup_{M_\ell \cap C_{\frac{R}{2}}}g_\ell \leq c(n) E_{M_\ell}(R)$ obtained in Theorem \ref{thm:immersion_sing_sheeting}, it follows that $\sup_{M_\ell \cap C_{\frac{1}{2}}}g_\ell \to 0$. If any other blow up gave rise to a different cone, we would have the existence of $y_\ell \in M_\ell \cap C_{\frac{1}{2}}$ with $y_\ell \to 0$ and $\limsup_{\ell \to \infty} g_\ell(y_\ell) >0$, contradiction.
\end{proof}

If $U=\R^{n+1}$ and the mass growth of $M$ at infinity is Euclidean, then tangents at infinity exist and are cones. The same argument shows:

\begin{cor}[Bernstein-type theorem]
\label{cor:uniq_tang_infinity}
Let $M$ be as in the beginning of Part \ref{partII} with $U=\R^{n+1}$, and assume that one tangent cone to $M$ at infinity (in the sense of varifolds) is a hyperplane with multiplicity. Then $M$ is a union of hyperplanes, parallel to the given tangent. (In particular, the tangent at infinity is unique.)
\end{cor}

\begin{proof}
Assume without loss of generality that a tangent at infinity is the hyperplane $\{x_{n+1}=0\}$ with multiplicity $q\in \N$. Then, by the monotonicity formula and by Remark \ref{oss:height_excess} there exists a sequence $R_\ell \to \infty$ such that $E_M(R_\ell)=\frac{1}{R_\ell^n} \int_{M \cap B_{R_\ell}^{n+1}(0)} |\nabla x_{n+1}|^2 \to 0$ as $\ell \to \infty$ (where $\nu$ is a chosen unit normal to the immersed hypersurface and $\nabla$ denotes the intrinsic gradient). For all sufficiently large $\ell$ we can therefore apply Theorem \ref{thm:eps_reg_tilt} to conclude that $\sup_{M \cap B_{\frac{R_\ell}{2}}^{n+1}(0)} |\nabla x_{n+1}| \leq c(n) E_M(R_\ell)$. For any given $r>0$ (since $B_r^{n+1}(0)  \subset B_{\frac{R_\ell}{2}}(0)$ for all sufficiently large $\ell$) one must thus have $\sup_{M \cap B_r^{n+1}(0) }|\nabla x_{n+1}| =0$. As $r$ is arbitrary, $\nabla x_{n+1} \equiv 0$ on $M$ and the conclusion follows.
\end{proof}

\begin{oss} 
If the multiplicity of the hyperplane is at most $2$, then the two corollaries follow from \cite{WicDG}.  
\end{oss}

\medskip

\textbf{Acknowledgments}. I wish to thank Otis Chodosh and Paul Minter for fruitful and helpful comments on an earlier version of the manuscript.

\appendix
\addcontentsline{toc}{part}{Appendices}

\section{The case $n=3$ of Theorem \ref{thm:eps_reg_A}}
\label{n=3}

While not essential for our arguments, we note explicitly that when $n=3$ a stronger conclusion in Theorem \ref{thm:eps_reg_A} can be obtained from the proof given:
\begin{cor}[$n=3$]
\label{cor:case_n=3}
Let $M$ be a properly immersed two-sided stable minimal hypersurface in $B_{2R}(0)$, with $0 \in M$. There exists an (explicit) increasing continuous function $y:[0,\infty)\to [0,\infty)$ with $y(0)=0$, such that for every $x\in M\cap B_{R/2}(0)$ we have
\[|A|(x)\leq \frac{y\Bigg( \frac{1}{2R} \int_{B_{2R}} |A|^2 \Bigg)}{R}.\]
\end{cor}

\begin{oss}
The proof gives $y(a) \sim a$ for $a$ large and $y(a) \sim \sqrt{a}$ for $a$ small.
\end{oss}

\begin{proof}
Repeating the proof of Theorem \ref{thm:eps_reg_A} with $n=3$ until the choice of $d$, and noting that $\frac{1}{R^2 d^{\frac{4}{n-2}}} + d^{\frac{2(n-4)}{n-2}} = \frac{1}{R^2 d^{4}} + \frac{1}{d^{2}}$, if we let $d=\frac{x}{R}$ the decay relation becomes
\[S_{\ell+1} \leq  c\Big( \frac{1}{x^4}+\frac{1}{x^2} \Big)^3 R^6 \tilde{C}^\ell S_{\ell-1}^{3},\]
with $\tilde{C}=2^{21}$ and $c=3^3 \cdot 2^{48}$ (using rough estimates, among which $C_{MS}^3\leq 4^{28}$).
A sufficient smallness condition on $S_1$ (to have $S_\ell \to 0$) is then
\[R^3 S_1= R^3 \int_{M\cap B_R} |A|^6 \leq \Big(\frac{x^4}{1+x^2}\Big)^{\frac{3}{2}}\frac{1}{c^{\frac{1}{2}} \tilde{C}^{\frac{3}{2}}}.\]
This is in turn implied\footnote{We use $R^3 \int_{B_R} |A|^6 \leq (40\,C_{MS})^3 \big(\frac{1}{2R} \int_{B_{2R}\cap M}|A|^2 \big)^3$, obtained in \eqref{eq:higher_integrability}.}, writing $K = c^{\frac{1}{3}}\tilde{C} (40\,C_{MS})^2$, by
\[ K  \Bigg( \frac{1}{2R} \int_{M\cap B_{2R}} |A|^2 \Bigg)^2 \leq \frac{x^4}{1+x^2}.\]
As $\frac{x^4}{1+x^2}$ is monotonically (strictly) increasing with value $0$ at $0$, we let $f$ denote its inverse function and set $y(a)=f(K  a^2)$. Then by choosing $d=\frac{y(a)}{R}$, with $a= \frac{1}{2R} \int_{B_{2R}} |A|^2$, we find $|A|\leq \frac{y(a)}{R}$ on $B_{\frac{R}{2}}$.
\end{proof}

\begin{oss}
\label{oss:n=3}
In other words, for $n=3$ curvature estimates of Theorem \ref{thm:curv_est_6} completely follow from Corollary \ref{cor:case_n=3} (without appealing to tangent cone analysis and dimension reducing). Indeed, $\frac{1}{(2R)^{n-2}} \int_{B_{2R}}|A|^2 \leq \omega_n 2^n \Lambda$ (by the stability inequality in $B_{4R}$), hence $|A|(x)\leq \frac{y(8 \omega_3 \Lambda)}{R}$ for every $x\in B_{\frac{R}{2}}$. 
More precisely, as $y$ above is explicit, for $n=3$ we find, 
in Theorem \ref{thm:curv_est_6}, $\beta = \sqrt{\frac{\sigma+\sqrt{\sigma^2+4\sigma}}{2}}$ with $\sigma=K(8 \omega_3 \Lambda)^2$ and $K$ as above. 

We remark that (always for $n=3$) \cite{ChoLi2} establishes the existence of a constant, explicitly determinable in terms of the first Betti number of $M$ and of the number of boundary components of $M$, that bounds the curvature in $B_{\frac{R}{2}}$ of Theorem \ref{thm:curv_est_6} .
\end{oss}

\section{An elementary lemma}

\begin{lem}
 \label{lem:analysis1}
Let $\tilde{C}, C>0, \alpha >0$ be given constants, and let $x_\ell$ be a sequence of positive real numbers that satisfies the following recursive relation for all $\ell\in \N \setminus \{0\}$:
\[x_{\ell+1} \leq \tilde{C}\, C^\ell \, x_\ell^{1+\alpha}.\]
Assume that $x_1 \leq \frac{1}{(\tilde{C} \, C^{1+\frac{1}{\alpha}})^{\frac{1}{\alpha}}}$ if $C>1$, $x_1 < \frac{1}{\tilde{C}^{\frac{1}{\alpha}}}$ if $C\leq 1$. Then $x_\ell \to 0$ as $\ell \to \infty$.
\end{lem}

\begin{proof}
Assume that $C>1$. We show that there exists $a\in(0,1)$ such that $\tilde{C}\, C^\ell \, x_\ell^{\alpha} \leq a$ for all $\ell\in \N$, from which $x_{\ell+1} \leq a x_\ell$ follows (hence the conclusion). For $\ell =1$ we have
\[\tilde{C} C x_1^{\alpha} \leq  \frac{C\, \tilde{C}}{\tilde{C}\, C^{1+\frac{1}{\alpha}}} =  \frac{1}{C^{\frac{1}{\alpha}}}\]
and we set $a= \frac{1}{C^{\frac{1}{\alpha}}}$. Now we check inductively, for arbitrary $(\ell+1)\geq 2$, that
\[\tilde{C}\, C^{\ell+1} \, x_{\ell+1}^{\alpha} \leq \tilde{C}\, C^{\ell+1} (\tilde{C}\, C^\ell \, x_\ell^{1+\alpha})^{\alpha}= C (\tilde{C} C^\ell x_\ell^\alpha)^{1+\alpha} \leq C a^{1+\alpha}=\frac{C}{C^{\frac{1+\alpha}{\alpha}}}=a.\]
 
If $C\leq 1$ then the recursive relation implies $x_{\ell+1} \leq \tilde{C}\, x_\ell^{1+\alpha}\leq (\tilde{C}\, x_\ell^{\alpha}) x_\ell$, in which case the smallness assumption $x_1 < \frac{1}{\tilde{C}^{\frac{1}{\alpha}}}$ implies the conclusion.
\end{proof}

\medskip

\medskip

\noindent UNIVERSITY COLLEGE LONDON

\noindent Gower Street, London, WC1E 6BT, United Kingdom

\noindent email: c.bellettini@ucl.ac.uk

\end{document}